\documentclass[10pt]{amsart}
\usepackage{amssymb,latexsym,pstricks}
\usepackage[normalem]{ulem}
\usepackage{orcidlink}
\topskip  \usepackage{color}
\usepackage{hyperref}
\usepackage{tensor}
\usepackage{enumitem}
\usepackage{mathabx}
\usepackage{comment}
\usepackage{amsfonts,amssymb,amsmath,amsthm,graphicx,latexsym}
\usepackage{mathrsfs,dsfont}
\usepackage{pstricks,pst-node,pst-plot}
\usepackage{pgfplots}
\usepackage{tikz}
\usepackage{bm}
\usepackage{caption}
\theoremstyle{thmstyleone}%
\newtheorem{theorem}{Theorem}[section]
\newtheorem{proposition}[theorem]{Proposition}%
\newtheorem{corollary}[theorem]{Corollary}

\theoremstyle{thmstyletwo}%
\newtheorem{example}[theorem]{Example}%
\newtheorem{remark}[theorem]{Remark}%

\theoremstyle{thmstylethree}%
\newtheorem{definition}[theorem]{Definition}%
\newtheorem{conjecture}[theorem]{Conjecture}%

\newcommand{\tp}{\mathbin{\hbox{$\bigcirc$\hbox to 0pt{\hspace{-0.81em}$\scriptstyle\top$\hfil}}}}

\begin{document}
\title[Normal ordering in the $(p,q)$-deformed generalized Weyl algebra]{Normal ordering in the $(p,q)$-deformed generalized Weyl algebra. III: The binomial formula}
\author[T. Mansour]{Toufik Mansour $^{1, \orcidlink{0000-0001-8028-2391}}$}
\address{$^1$ Department of Mathematics, University of Haifa, 3498838 Haifa, Israel, tmansour@univ.haifa.ac.il}
\author[L. Oussi]{Lahcen Oussi$^{2, \orcidlink{0000-0001-9804-0761}}$}
\address{$^2$ Faculty of Pure and Applied Mathematics, Wrocław University of Science and Technology, Wybrzeże Wyspiańskiego 27, 50-370 Wrocław, Poland, lahcen.oussi@pwr.edu.pl, oussimaths@gmail.com}
\author[M. Schork]{Matthias Schork $^{3, \orcidlink{0000-0002-9759-6618}}$}
\address{$^3$ Institute for Mathematics, Würzburg University, Emil-Fischer Str. 40, 97074 Würzburg, Germany, matthias.schork@mathematik.uni-wuerzburg.de}


\begin{abstract}
We study the $(p, q)$-deformed generalized Weyl algebra generated by variables $X, Y$ and $Z_p$ satisfying the $(p, q)$-commutation relations $XY-qYX=h Y^sZ_{p}, XZ_p=pZ_pX$, and $Z_pY=pYZ_p$, with $s\in \mathbb{N}_0$. Within this framework, we investigate the noncommutative binomial formula $(X+Y)^n$ and related identities. In particular, we show how the associated normal ordering coefficients can be expressed in terms of $(p,q)$-deformed $s$-rook numbers. We treat several special cases explicitly, recovering known results from literature as well as deriving new ones.
\end{abstract}
\maketitle
\section{Introduction}\label{sec1}
The present paper is a continuation of the systematic investigation initiated by the authors in \cite{TLM1,TLM2}. In  \cite{TLM1}, we established the algebraic framework and derived various combinatorial identities within the context of the $(p, q)$-\emph{deformed generalized Weyl algebra}. Subsequently, in \cite{TLM2},  we provided a combinatorial interpretation of the corresponding normal ordering process in term of rook placements on specialized boards. As a natural extension of these normal ordering results, we turn our attention in the present paper to binomial expansions. The ultimate starting point for all such expansions is the classical Newton's binomial formula,
\begin{equation}\label{CBF}
	(x+y)^n=\sum_{k=0}^{n}\binom{n}{k}y^{k}x^{n-k} ,
\end{equation} 
where the two variables $x$ and $y$ commute (hence, their order on the right-hand side is irrelevant). However, in the noncommutative setting, formula \eqref{CBF} does not hold true and one has to fix the order of the letters on the right-hand side. The noncommutative binomial formula $(U+V)^n$ was studied by several authors \cite{HBB1998, HBB1999, AHMM2002, HJYZ2023, TMMS2016, HR2000, MS2021, OVV1998}, depending on the commutation relation satisfied by the variables $U$ and $V$. For instance, Benaoum \cite{HBB1998} derived the binomial formula $(U+V)^n=\sum_{k=0}^{n}\binom{n}{k}_{h}V^kU^{n-k}$ for variables $U$ and $V$ satisfying $UV-VU=hV^2$ (i.e., the {\em Jordan plane}), where $\binom{n}{k}_{h}$ denotes {\em $h$-binomial coefficients}.  Moreover, Benaoum \cite{HBB1999}, derived also a $q$-deformation of this formula as 
$$(U+V)^n=\sum_{k=0}^{n}\binom{n}{k}_{q,h}V^kU^{n-k},$$ 
where the variables $U$ and $V$ satisfy the commutation relation of the {\em $q$-deformed Jordan plane},
\begin{equation}\label{qhcr}
	UV-qVU=hV^2,
\end{equation}
and where $\binom{n}{k}_{q,h}$ denotes a {\em $(q,h)$-binomial coefficient} which reduces for $q=1$ to the $h$-binomial coefficient. For $h=0$, \eqref{qhcr} leads to $UV=qVU$, i.e., the $q$-commutation relation of the {\em quantum plane} \cite{YIM1988}, and the related $q$-binomial formula for the quantum plane is a famous result due to Potter \cite{HSAP1950} and Schützenberger \cite{MPS1953},
\begin{equation}\label{BFPS}
	(U+V)^n=\sum_{k=0}^{n}\binom{n}{k}_qV^kU^{n-k},
\end{equation} 
where $\binom{n}{k}_q$ denotes {\em $q$-binomial coefficients} (or, {\em Gaussian binomial coefficients}). It is the natural $q$-analog to \eqref{CBF}. As a last example, let us mention that the binomial formula \eqref{Yamazaki} in the Weyl algebra was derived by Yamazaki in a physical context already in 1952, and it was given a combinatorial interpretation in terms of rook placements by Varvak in \cite{AV2005}, where also associated {\em Weyl binomial coefficients} were introduced. Mansour and Schork \cite{TMMS2011} studied the general case of the above mentioned formulas by considering variables $U$ and $V$ which satisfy the $q$-commutation relation $UV=qVU+hf(V)$, where $f$ is a polynomial, and investigated the binomial formula $(U+V)^n$ using Young diagrams and generating function techniques. Several special cases were treated explicitly, depending on the choice of parameters $h, q$ and polynomial $f$.  Most concrete results were derived for the case that $f$ is a monomial, $f(V)=V^s$ with $s\in \mathbb{N}_0:=\mathbb{N}\cup\{0\}$.

In the present paper, we consider variables $X, Y$ and $Z_p$ satisfying the following $(p,q)$-commutation relations, 
\begin{align}
	XY-qYX &= hf(Y)Z_p,\label{pqcr1} \\
	Z_pY&=pYZ_p, \label{pqcr2}\\
	XZ_p&=pZ_pX, \label{pqcr3}\\
	Z_1 &=I,\label{pqcr4}
\end{align}
where $f$ denotes a polynomial function. Within this framework, a word $\omega$ composed of letters $X,Y$ and $Z_p$ is considered to be in \emph{normal ordered} form if all the letters $Y$ stand to the left of all the letters $Z_p$ which stand to the left of all the letters $X$. More precisely, the normal ordered expansion takes the form $\omega=\sum_{k,l,m} a_{k,l,m}Y^kZ_p^lX^m$ for a given set of coefficients $a_{k,l,m}$, the {\em normal ordering coefficients}. While in our works \cite{TLM1,TLM2} we established the algebraic framework for normal ordering and an interpretation in terms of rook placements for these relations, the primary objective of the current work is to study the binomial formula for the expansion of $(X+Y)^n$ under the specific choice of the monomial $f(Y)=Y^s$, i.e., in the {\em $(p, q)$-deformed generalized Weyl algebra}. By varying and specializing the parameters $s,h,p$, and $q$, we analyze special cases, thereby recovering numerous results found throughout the literature as well as deriving new ones. For instance,  setting $p=1$ yields the $q$-commutation relation $XY-qYX = hf(Y)$ which was studied in~\cite{CCG2017,TMMS2011}. Additionally, by letting $q=1$, we recover the commutation relation investigated in~\cite{BLO2013,BLO2015,BLO2015a}. On the other hand, by specializing $X=D_{p, q}$ to be the $(p, q)$-derivative operator, $Y$ to be the operator of multiplication with $x$, and $Z_p$ to be the Fibonacci operator \cite{GCMDCH}, one obtains the specific $(p, q)$-commutation relations studied in \cite{LO2021}. \

The structure of this paper is organized as follows. In Section~\ref{secprel}, we provide a brief review and background regarding normal ordering in the $(p, q)$-deformed generalized Weyl algebra, along with its interpretation in terms of rook placements. Section~\ref{bfm} contains our main results, where we derive the explicit binomial formula for $(X+Y)^n$ within this $(p, q)$-deformed generalized Weyl algebra. Furthermore, we investigate numerous special cases by varying the parameters $s, h, p$, and $q$. Finally, in Section~\ref{concl}, we present some conclusions.

\section{Preliminaries}\label{secprel}
In this section, we review some basic tools and background regarding  normal ordering within the $(p,q)$-deformed generalized Weyl algebra and the associated rook numbers from \cite{TLM1,TLM2}.

For any non-negative integer $n\in\mathbb{N}_0$, we define the {\em twin-basic numbers} (viewed as a $(p,q)$-deformation of the natural numbers) by
$$
[n]_{p, q}:=\frac{p^n-q^n}{p-q}=\sum_{k=1}^{n}p^{n-k}q^{k-1}.
$$ 
The associated $(p,q)$-factorials and $(p, q)$-binomial coefficients are given, respectively, by
$$
[n]_{p, q}!=[n]_{p, q}[n-1]_{p, q}\cdots [2]_{p, q}[1]_{p, q},\,\,\, \text{with }\, [0]_{p, q}!=1,
$$
and 
$${n\choose k}_{p, q}=\frac{[n]_{p, q}!}{[n-k]_{p, q}![k]_{p, q}!}.$$ 
Furthermore, the corresponding {\em $(p,q)$-derivative} operator $D_{p,q}$ acts on a given function $f(x)$ as
$$
(D_{p,q}f)(x):=\frac{f(px)-f(qx)}{(p-q)x}.
$$
The action of $D_{p,q}$ on monomials is given, for $n \in \mathbb{N}$, by 
$$
D_{p,q}x^n=[n]_{p, q}x^{n-1}.
$$
In the limiting case where  $p=1$, we recover the $q$-deformation of the above definitions. We refer the reader to~\cite{LO2024} and the references therein for more details.

Let us recall the following definition from~\cite{TLM1}. 
\begin{definition}\label{DefpqGenWeyl}
	The {\em $(p, q)$-deformed generalized Weyl algebra $A_{s;h|p,q}$} is defined, for $s\in\mathbb{N}_{0}, h\in\mathbb{C}\setminus\{0\}$ and $p, q\in\mathbb{C}$, as the complex unital algebra generated by variables $X,Y$ and $Z_p$ satisfying 
	\begin{equation}\label{pqdefgenWeyl}
		XY-qYX = h Y^sZ_p, \,\, Z_pY=pYZ_p, \,\, XZ_p=pZ_pX, \,\, Z_1 =I.
	\end{equation}
\end{definition}
It follows that, for $p=1$, the above definition reduces to the $q$-deformed generalized Weyl algebra $A_{s;h|1,q}=A_{s;h|q}$ \cite[Definition 1.25]{TMMS2016}, which  was first introduced in \cite{TMMS2011}. On the other hand, the case $p=q=h=1$ was established by Burde \cite{Burde2005} (in the context of matrices $X$ and $Y$), and Varvak \cite{AV2005} (see also \cite[Section 9.1]{TMMS2016}). 

We consider normal ordering arbitrary words  in $X$ and $Y$ which satisfy the commutation relations \eqref{pqdefgenWeyl} of the $(p, q)$-deformed generalized Weyl algebra $A_{s;h|p,q}$. Such a word can be written in the form
\begin{equation}\label{wordxy}
	\mathrm{w}_{{\bf m}, {\bf n}}=Y^{n_r}X^{m_r}\cdots Y^{n_1}X^{m_1},
\end{equation}
with suitable exponents $m_k,n_k \in \mathbb{N}_0$, and ${\bf m}:=(m_1, \ldots, m_r), {\bf n}:=(n_1, \ldots, n_r)$. The resulting normal ordering coefficients can be interpreted as generalized Stirling numbers. Namely, within this context, for variables $X, Y$ and $Z_p$ satisfying the commutation relations \eqref{pqdefgenWeyl} of the $(p, q)$-deformed generalized Weyl algebra $A_{s;h|p,q}$, the $(p, q)$-deformed generalized Stirling numbers $S_{s, h; p, q}^{{\bf m}, {\bf n}}[k]$ are defined in \cite{TLM1} as normal ordering coefficients of the string $\mathrm{w}_{{\bf m}, {\bf n}}$ as follows,
\begin{equation}\label{pqrook}
	\mathrm{w}_{{\bf m}, {\bf n}}=\sum_{k=m_1}^{|{\bf m}|}S_{s, h; p, q}^{{\bf m}, {\bf n}}[k]Y^{|{\bf n}|-(|{\bf m}|-k)(1-s)}Z_{p}^{|{\bf m}|-k}X^{k}.
\end{equation}
For the special case  ${\bf m}={\bf 1}_r=(
\raisebox{0pt}[\height][0pt]{$\underbrace{1, 1\ldots1}_{r\ \text{times}})$}$ and ${\bf n}={\bf 1}_r=(
\raisebox{0pt}[\height][0pt]{$\underbrace{1, 1\ldots1}_{r\ \text{times}})$}$, \vspace{.5cm} we obtain the {\it generalized Stirling numbers} $\mathfrak{S}_{s;h}(n,k|p,q)$ introduced in \cite[Definition 3.3]{LO2021} as the normal ordering coefficients of $(YX)^n$ in $A_{s;h|p,q}$ via the expansion
\begin{equation}\label{StirGen}
	(YX)^n=\sum_{k=0}^n \mathfrak{S}_{s;h}(n,k|p,q) \, Y^{s(n-k)+k}Z_p^{n-k} X^{k}.
\end{equation}

Furthermore, normal ordering can also be interpreted using rook numbers. To  see this, we first recall the basic definitions of rook theory from \cite{TLM2,TMMS2011}.

By a {\em partition} $\lambda=(\lambda_1,\lambda_2,\ldots,\lambda_k) \equiv \lambda_1\lambda_2\ldots\lambda_k$, we mean a weakly decreasing sequence of positive integers whose weight is give by  $|\lambda|=\sum_{i=1}^k\lambda_i$. Geometrically, for a partition $\lambda$, the {\em Young diagram} $Y_\lambda$ of shape $\lambda$ is a left-justified diagram of $|\lambda|$ boxes, with $\lambda_i$ black boxes in the $i$th column. Let ${\mathcal I}_{\ell,k}$ denote the collection of all Young diagrams that are contained in a $ \ell \times k$ box, and define the disjoint union ${\mathcal I}_{m}=\bigcup_{k=0}^m {\mathcal I}_{m-k,k}$. Since the cardinality of each component ${\mathcal I}_{m-k,k}$ is given by the binomial coefficient $\binom{m}{k}$, it follows that $|{\mathcal I}_m|=2^m$. The empty partition is represented by a sequence of zeros (such as  $000 \in {\mathcal I}_{1,3}$ or $00 \in {\mathcal I}_{2,2}$, see  Figure~\ref{fig1}), though it is otherwise denoted simply by $\emptyset$.
		
		\begin{example}The set $\mathcal{I}_4$ containing 16 partitions is given in Figure~\ref{fig1}.
			\begin{center}
				\begin{figure}[htp]
					\begin{pspicture}(14,2.2)
						\psset{xunit=0.7,yunit=0.7}
						\put(-.5,0){\psline(0,2.8)(1,2.8) \put(.2,2.15){$\mathcal{I}_{0,4}$}\put(0.8,1.9){$\lambda=0000$}}
						\put(2,0){\put(0,.6){\psline(0,2)(.75,2)\psline(0,1.75)(.75,1.75)\psline(.25,1.75)(.25,2)\psline(.5,1.75)(.5,2)\psline(.75,1.75)(.75,2)
								\psline(0,1.75)(0,2)}\put(0.8,2.15){$\mathcal{I}_{1,3}$}\put(.75,1.8){$\lambda=000$}}
						\put(2,-.5){\put(0,.6){\psline(0,2)(.75,2)\psline(0,1.75)(.75,1.75)\psline(.25,1.75)(.25,2)\psline(.5,1.75)(.5,2)\psline(.75,1.75)(.75,2)
								\psline(0,1.75)(0,2)\psline[fillstyle=solid,fillcolor=black](0,2)(.25,2)(.25,1.75)(0,1.75)(0,2)}\put(.75,1.8){$\lambda=100$}}
						\put(2,-1){\put(0,.6){\psline(0,2)(.75,2)\psline(0,1.75)(.75,1.75)\psline(.25,1.75)(.25,2)\psline(.5,1.75)(.5,2)\psline(.75,1.75)(.75,2)
								\psline(0,1.75)(0,2)\psline[fillstyle=solid,fillcolor=black](0,2)(.5,2)(.5,1.75)(0,1.75)(0,2)}\put(.75,1.8){$\lambda=110$}}
						\put(2,-1.5){\put(0,.6){\psline(0,2)(.75,2)\psline(0,1.75)(.75,1.75)\psline(.25,1.75)(.25,2)\psline(.5,1.75)(.5,2)\psline(.75,1.75)(.75,2)
								\psline(0,1.75)(0,2)\psline[fillstyle=solid,fillcolor=black](0,2)(.75,2)(.75,1.75)(0,1.75)(0,2)}\put(.75,1.8){$\lambda=111$}}
						\put(4.5,0){\put(0,.6){\psline(0,2)(.5,2)(.5,1.5)(0,1.5)(0,2)\psline(0,1.75)(.5,1.75)\psline(.25,1.5)(.25,2)}
							\put(1,2.15){$\mathcal{I}_{2,2}$}\put(.5,1.7){$\lambda=00$}}
						\put(4.5,-.75){\put(0,.6){\psline(0,2)(.5,2)(.5,1.5)(0,1.5)(0,2)\psline(0,1.75)(.5,1.75)\psline(.25,1.5)(.25,2)
								\psline[fillstyle=solid,fillcolor=black](0,2)(.25,2)(.25,1.75)(0,1.75)(0,2)}\put(.5,1.7){$\lambda=10$}}
						\put(4.5,-1.5){\put(0,.6){\psline(0,2)(.5,2)(.5,1.5)(0,1.5)(0,2)\psline(0,1.75)(.5,1.75)\psline(.25,1.5)(.25,2)
								\psline[fillstyle=solid,fillcolor=black](0,2)(.5,2)(.5,1.75)(0,1.75)(0,2)}\put(.5,1.7){$\lambda=11$}}
						\put(6.25,0){\put(0,.6){\psline(0,2)(.5,2)(.5,1.5)(0,1.5)(0,2)\psline(0,1.75)(.5,1.75)\psline(.25,1.5)(.25,2)
								\psline[fillstyle=solid,fillcolor=black](0,2)(.25,2)(.25,1.5)(0,1.5)(0,2)}\put(.5,1.7){$\lambda=20$}}
						\put(6.25,-.75){\put(0,.6){\psline(0,2)(.5,2)(.5,1.5)(0,1.5)(0,2)\psline(0,1.75)(.5,1.75)\psline(.25,1.5)(.25,2)
								\psline[fillstyle=solid,fillcolor=black](0,2)(.5,2)(.5,1.75)(.25,1.75)(.25,1.5)(0,1.5)(0,2)}\put(.5,1.7){$\lambda=21$}}
						\put(6.25,-1.5){\put(0,.6){\psline(0,2)(.5,2)(.5,1.5)(0,1.5)(0,2)\psline(0,1.75)(.5,1.75)\psline(.25,1.5)(.25,2)
								\psline[fillstyle=solid,fillcolor=black](0,2)(.5,2)(.5,1.5)(0,1.5)(0,2)}\put(.5,1.7){$\lambda=22$}}
						\put(8.5,0){\put(0,.6){\psline(0,2)(.25,2)(.25,1.25)(0,1.25)(0,2)\psline(0,1.75)(.25,1.75)\psline(0,1.5)(.25,1.5)}
							\put(1,2.15){$\mathcal{I}_{3,1}$}\put(.25,1.5){$\lambda=0$}}
						\put(8.5,-1){\put(0,.6){\psline(0,2)(.25,2)(.25,1.25)(0,1.25)(0,2)\psline(0,1.75)(.25,1.75)\psline(0,1.5)(.25,1.5)
								\psline[fillstyle=solid,fillcolor=black](0,2)(.25,2)(.25,1.75)(0,1.75)(0,2)}\put(.25,1.5){$\lambda=1$}}
						\put(10.25,0){\put(0,.6){\psline(0,2)(.25,2)(.25,1.25)(0,1.25)(0,2)\psline(0,1.75)(.25,1.75)\psline(0,1.5)(.25,1.5)
								\psline[fillstyle=solid,fillcolor=black](0,2)(.25,2)(.25,1.5)(0,1.5)(0,2)}\put(.25,1.5){$\lambda=2$}}
						\put(10.25,-1){\put(0,.6){\psline(0,2)(.25,2)(.25,1.25)(0,1.25)(0,2)\psline(0,1.75)(.25,1.75)\psline(0,1.5)(.25,1.5)
								\psline[fillstyle=solid,fillcolor=black](0,2)(.25,2)(.25,1.25)(0,1.25)(0,2)}\put(.25,1.5){$\lambda=3$}}
						\put(12.7,0){\put(0,.6){\psline(0,2)(0,1)}\put(-.3,2.15){$\mathcal{I}_{4,0}$}\put(.1,1.25){$\lambda=0$}}
					\end{pspicture}
					\caption{The Young diagrams in $\mathcal{I}_4$.}\label{fig1}
				\end{figure}
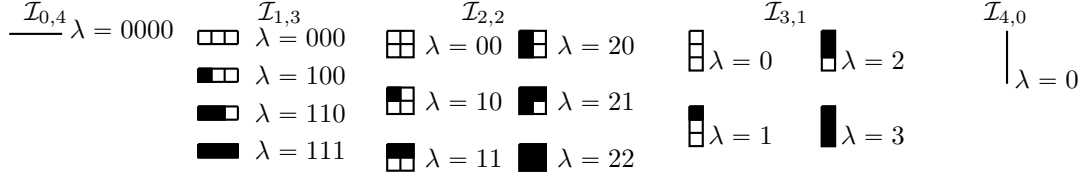
			\end{center}
		\end{example}

By aligning the columns at the top, we identify the Young diagram $Y_{\lambda}$ with the Ferrers board denoted by $B_{\lambda}$. A {\em file placement of $k$ rooks} (or  $k$-{\em file placement}) on a Ferrers board $B$ is a placement of $k$ rooks  containing at most one rook in each column. We denote the set of all such placements by $\mathcal{F}_k(B)$ and its cardinality $f_k(B)=|\mathcal{F}_k(B)|$ is the $k$-{\em th file number of $B$}. A $k$-{\em rook placement} is a special kind of file placement where in addition no two rooks are in the same row (i.e., the $k$ rooks are {\it non-attacking}). The set of all $k$-rook placements on $B$ will be denoted by $\mathcal{R}_k(B)$, and $r_k(B)=|\mathcal{R}_k(B)|$ is the $k$-{\em th rook number of $B$}. Figure~\ref{FigFile} illustrates a $4$-file placement on the Ferrers board $B_{33211}$ corresponding to the partition $\lambda=33211$.
		
		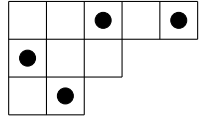
\begin{figure}[h]
			\centering
			\begin{tikzpicture}
				
				\draw (0,0.5) -- (1,0.5);
				
				\draw (0.5,0.5) -- (0.5,1.5);
				\draw (1,0.5) -- (1,2);  
				\draw (1.5,1) -- (1.5,2); 
				\draw (2,1.5) -- (2,2);
				\draw (2.5,1.5) -- (2.5,2);  
				
				\draw (0,0.5) -- (0,2);
				\draw (0.5,1) -- (0.5,2); 
				\draw (0,1) -- (1.5,1); 
				\draw (0,1.5) -- (2.5,1.5); 
				\draw (0,2) -- (2.5,2);
				
				\filldraw (0.25,1.25) circle (3pt);
				\filldraw (0.75,0.75) circle (3pt);
				\filldraw (1.25,1.75) circle (3pt);
				\filldraw (2.25,1.75) circle (3pt);
			\end{tikzpicture}
			\caption{The Ferrers board $B_{33211}$ with a $4$-file placement which is not a $4$-rook placement.}\label{FigFile}
		\end{figure}
	
		This model of rook and file placements was generalized by the authors in~\cite{TLM2}, following the approach of Goldman and Haglund \cite{JGJH2000}. 
		Let $i\in \mathbb{N}_0$ and let $B$ be a Ferrers board. The {\em $i$-row creation rule} means that placing a rook on $B$ introduces $i$ new rows to its left (see,  Figure~\ref{FigFile2}). We place the rooks on the Board $B$ from right to left.  When creating $i$ new rows, we interpret the resulting top row as the row belonging to the rook just placed while the $i$ rows beneath it as new rows.
		\begin{figure}[h]
			\centering
			\begin{tikzpicture}
				
				\draw (0,0.5) -- (1,0.5);
				
				\draw (0.5,0.5) -- (0.5,1.5);
				\draw (1,0.5) -- (1,2);  
				\draw (1.5,1) -- (1.5,2); 
				\draw (2,1.5) -- (2,2);
				\draw (2.5,1.5) -- (2.5,2);  
				
				\draw[dashed] (0,0.75) -- (0.5,0.75);
				\draw[dashed] (0,1.75) -- (2.0,1.75);
				\draw[dashed] (0,1.62) -- (1,1.62);
				
				\draw (0,0.5) -- (0,2);
				\draw (0.5,1) -- (0.5,2); 
				\draw (0,1) -- (1.5,1); 
				\draw (0,1.5) -- (2.5,1.5); 
				\draw (0,2) -- (2.5,2);
				
				\filldraw (0.75,0.85) circle (1.5pt);
				\filldraw (1.25,1.65) circle (1.5pt);
				\filldraw (2.25,1.85) circle (1.5pt);
			\end{tikzpicture}
			\caption{The Ferrers board $B_{33211}$ with a $3$-rook placement with 1-row creation rule.}\label{FigFile2}
		\end{figure}
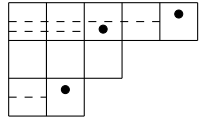 
		
		\begin{definition}[\cite{JGJH2000}]\label{SRook} Let $i\in \mathbb{N}_0$. Given a Ferrers board $B$, the {\em $k$-th $i$-rook number} $r_k^{(i)}(B)$ is the number of ways to place $k$ non-attacking rooks on the board $B$ going from right to left, creating $i$ new rows to the left of each rook.  
		\end{definition}
		Clearly, for $i=0$ one recovers the conventional rook numbers. For $i=1$ (Figure~\ref{FigFile2}), placing a rook creates a new row to its left. Thus, for the next rook to place, in each column to the left one has the same amount of boxes to choose from as the column has. This implies that a $k$-rook placement under the $1$-row creation rule corresponds exactly to a $k$-file placement. Thus,  
		\begin{equation}\label{RookRed}
			r_k^{(0)}(B)=r_k(B), \,\,\,\, r_k^{(1)}(B)=f_k(B). 
		\end{equation}
 
Following \cite{TLM2}, we now introduce a $(p,q)$-weight for rook placements that extends the case $p=1$ (for $s=0$). Let $\phi$ be a placement of $k$ non-attacking rooks on the Ferrrers board $B$. We partition the boxes of $B$ into 4 disjoint classes:
\begin{enumerate}
	\item A {\em rook box} is a box of $B$ occupied by a rook.
	\item A {\em $p$-box} is a box which is not a rook box and which is lying above or to the left of a rook, and
	\begin{itemize}
		\item it is not lying to the left of another rook, and
		\item it is not lying in the column of another rook.
	\end{itemize}
	\item A {\em cancelled box} is a box located above or to the left of a rook which is not a $p$-box.
	\item All remaining boxes of $B$ are {\em $q$-boxes}.
\end{enumerate} 

The {\em $(p,q)$-deformed rook numbers} are defined in \cite{TLM2} as follows.
\begin{definition}
	Let $B$ be a Ferrers board and let $\mathcal{R}(B, k)$ be the set of placements of $k$ non-attacking rooks on $B$. Then the {\em $(p,q)$-deformed rook numbers} are defined by
	\begin{equation}\label{PQRookWe}
		\mathscr{R}_{h; p,q}[B, k] :=\sum\limits_{\phi\in\mathcal{R}(B, k)} \omega_{p,q}(B;\phi)=h^k \sum\limits_{\phi\in\mathcal{R}(B, k)} p^{\# p-boxes}q^{\# q-boxes},
	\end{equation}
where 
\begin{equation}\label{PQWeight}
	\omega_{p,q}(B;\phi) :=h^k p^{\# p-boxes}q^{\# q-boxes},
\end{equation}
is the {\em $(p,q)$-weight} of a placement $\phi$ of $k$ non-attacking rooks on the Ferrers board $B$ .	
\end{definition}
In this setting, a $q$-box is simply any box that does not contain a rook and has no rooks in its row to the right or in its column below. 

\begin{example} 
In the following figure, we mark a rook box by a rook, a $p$-box (respectively, $q$-box) with a $p$ (respectively, $q$), and the cancelled boxes with an $x$. In Figure~\ref{FigFilePQ} a 6-rook placement on a board can be seen. There are in total 68 boxes: 6 rook-boxes, 20 $p$-boxes, 15 cancelled boxes, and 27 $q$-boxes. Thus, the $(p,q)$-weight of this placement is given by $h^6 p^{20}q^{27}$.

\begin{figure}[h]
	\centering
	\begin{tikzpicture}
		
		\draw (0,6.5) -- (6,6.5);
		\draw (0,6) -- (6,6);
		\draw (0,5.5) -- (4.5,5.5);
		\draw (0,5) -- (4.5,5);
		\draw (0,4.5) -- (4,4.5); 
		\draw (0,4) -- (4,4); 
		\draw (0,3.5) -- (3,3.5); 
		\draw (0,3) -- (2.5,3); 
		\draw (0,2.5) -- (2,2.5); 
		\draw (0,2) -- (2,2); 
		\draw (0,1.5) -- (1,1.5); 
		\draw (0,1) -- (0.5,1);

		
		\draw (0,1) -- (0,6.5);
		\draw (0.5,1) -- (0.5,6.5);
		\draw (1,1.5) -- (1,6.5);
		\draw (1.5,2) -- (1.5,6.5);
		\draw (2,2) -- (2,6.5);
		\draw (2.5,3) -- (2.5,6.5);
		\draw (3,3.5) -- (3,6.5);
		\draw (3.5,4) -- (3.5,6.5);
		\draw (4,4) -- (4,6.5);
		\draw (4.5,5) -- (4.5,6.5);
		\draw (5,6) -- (5,6.5);
		\draw (5.5,6) -- (5.5,6.5);
		\draw (6,6) -- (6,6.5);

		\node at (0.25,6.25) {x};
		\node at (0.75,6.25) {x};
		\node at (1.25,6.25) {x};
		\node at (1.75,6.25) {$p$};
		\node at (2.25,6.25) {x};
		\node at (2.75,6.25) {$p$};
		\node at (3.25,6.25) {x};
		\node at (3.75,6.25) {$p$};
		\node at (4.25,6.25) {$p$};
		\node at (4.75,6.25) {$p$};
		\filldraw (5.25,6.25) circle (3pt);
		\node at (5.75,6.25) {$q$};

		\node at (0.25,5.75) {x};
		\node at (0.75,5.75) {x};
		\node at (1.25,5.75) {x};
		\node at (1.75,5.75) {$p$};
		\node at (2.25,5.75) {x};
		\node at (2.75,5.75) {$p$};
		\filldraw (3.25,5.75) circle (3pt);
		\node at (3.75,5.75) {$q$};
		\node at (4.25,5.75) {$q$};
	
		\node at (0.25,5.25) {$p$};
		\node at (0.75,5.25) {$p$};
		\node at (1.25,5.25) {$p$};
		\node at (1.75,5.25) {$q$};
		\node at (2.25,5.25) {$p$};
		\node at (2.75,5.25) {$q$};
		\node at (3.25,5.25) {$q$};
		\node at (3.75,5.25) {$q$};
		\node at (4.25,5.25) {$q$};
		
		\filldraw (0.25,4.75) circle (3pt);
		\node at (0.75,4.75) {$p$};
		\node at (1.25,4.75) {$p$};
		\node at (1.75,4.75) {$q$};
		\node at (2.25,4.75) {$p$};
		\node at (2.75,4.75) {$q$};
		\node at (3.25,4.75) {$q$};
		\node at (3.75,4.75) {$q$};
		
		\node at (0.25,4.25) {$q$};
		\node at (0.75,4.25) {$p$};
		\node at (1.25,4.25) {$p$};
		\node at (1.75,4.25) {$q$};
		\node at (2.25,4.25) {$p$};
		\node at (2.75,4.25) {$q$};
		\node at (3.25,4.25) {$q$};
		\node at (3.75,4.25) {$q$};
	
		\node at (0.25,3.75) {$x$};
		\node at (0.75,3.75) {$x$};
		\filldraw (1.25,3.75) circle (3pt);
		\node at (1.75,3.75) {$q$};
		\node at (2.25,3.75) {$p$};
		\node at (2.75,3.75) {$q$};
		
		\node at (0.25,3.25) {$x$};
		\node at (0.75,3.25) {$x$};
		\node at (1.25,3.25) {$x$};
		\node at (1.75,3.25) {$p$};
		\filldraw (2.25,3.25) circle (3pt);
		
		\node at (0.25,2.75) {$q$};
		\node at (0.75,2.75) {$p$};
		\node at (1.25,2.75) {$q$};
		\node at (1.75,2.75) {$q$};
		
		\node at (0.25,2.25) {$x$};
		\filldraw (0.75,2.25) circle (3pt);
		\node at (1.25,2.25) {$q$};
		\node at (1.75,2.25) {$q$};

		\node at (0.25,1.75) {$q$};
		\node at (0.75,1.75) {$q$};
		
		\node at (0.25,1.25) {$q$};

	\end{tikzpicture}
	\caption{A Ferrers board with a $6$-rook placement and marking of the boxes according to their class.}\label{FigFilePQ}
\end{figure}
	
\end{example}

On the other hand, for the case $s>0$, we recall the  $(p,q)$-weights for $s$-rook placements introduced in~\cite{TLM2}, which extend the work of Celeste et al. \cite{CCG2017} to $p \neq 1$. Let $\phi$ be a rook placement on a Ferrers board $B$ satisfying the $s$-row creation rule, and let $\mathcal{R}_s(B, k)$ denotes the set of all such placements of $k$ rooks. For any $\phi \in \mathcal{R}_s(B, k)$, we partition the boxes of $B$ into 5 disjoint classes:
\begin{enumerate}
	\item A {\em rook box} is a box of $B$ occupied by a rook.
	\item A {\em $t$-box} is a box which is not a rook box and which is lying above a rook and not lying to the left of another rook.
	\item An {\em $\ell $-box} is a box which is not a rook box and which is lying to the left of a rook and which is not lying in the column of another rook.
	\item A {\em cancelled box} is a box which is not a rook box and which is lying to the left of a rook and lying in the column of another rook.
	\item All remaining boxes of $B$ are {\em $e$-boxes} (for ``empty'').
\end{enumerate}

The {\em $(p,q)$-deformed $s$-rook numbers} are defined as follows \cite{TLM2}.
\begin{definition}
	Let $s>0$. Let $B$ be a Ferrers board and let $\mathcal{R}_s(B, k)$ be the set of placements of $k$ rooks on $B$ satisfying the $s$-row creation rule. Then the {\em $(p,q)$-deformed $s$-rook numbers} are defined by
	\begin{equation}\label{PQRookWeS}
		\mathscr{R}_{s,h; p,q}[B, k] :=\sum\limits_{\phi\in\mathcal{R}_s(B, k)} \omega_{s;p,q}(B;\phi)=h^k \sum\limits_{\phi\in\mathcal{R}_s(B, k)} p^{\# t-boxes + \# \ell-boxes}q^{\# e-boxes},
	\end{equation}
where
	\begin{equation}\label{PQWeightS}
		\omega_{s;p,q}(B;\phi) :=h^k p^{\# t-boxes + \# \ell-boxes}q^{\# e-boxes},
	\end{equation}
is the {\em $(p,q)$-weight} of a placement $\phi$ of $k$ rooks satisfying the $s$-row creation rule on the Ferrers board $B$.
\end{definition}

Clearly, $\mathscr{R}_{s,h; p,q}[B, k] =h^k \mathscr{R}_{s,1; p,q}[B, k] $. Observe that for $p=1$ the weight of a placement reduces to $\omega_{s;1,q}(B;\phi)=h^k q^{ \# e-boxes}$, i.e., the parameter $q$ counts the number of boxes not containing a rook and not lying above or to the left of a rook. This is exactly the parameter introduced in \cite{CCG2017}, i.e., $\mathscr{R}_{s,1; 1,q}[B, k] =R_{s,1;q}[B, k]$ defined there as $q$-deformed $s$-rook number.
\begin{example}
	Figure~\ref{FigFileSRook} illustrates a placement $\phi$ of $6$ rooks on a Ferrers board $B$ under the  $s=2$ row creation rule. The board consists of  104 boxes in total, partitioned into $6$ rook boxes, $32$ $t$-boxes, $11$ $\ell$-boxes, $15$ cancelled boxes and $40$ $e$-boxes. Consequently, the  $(p,q)$-weight of this placement $\phi$ is given by $h^6p^{43}q^{40}$.
	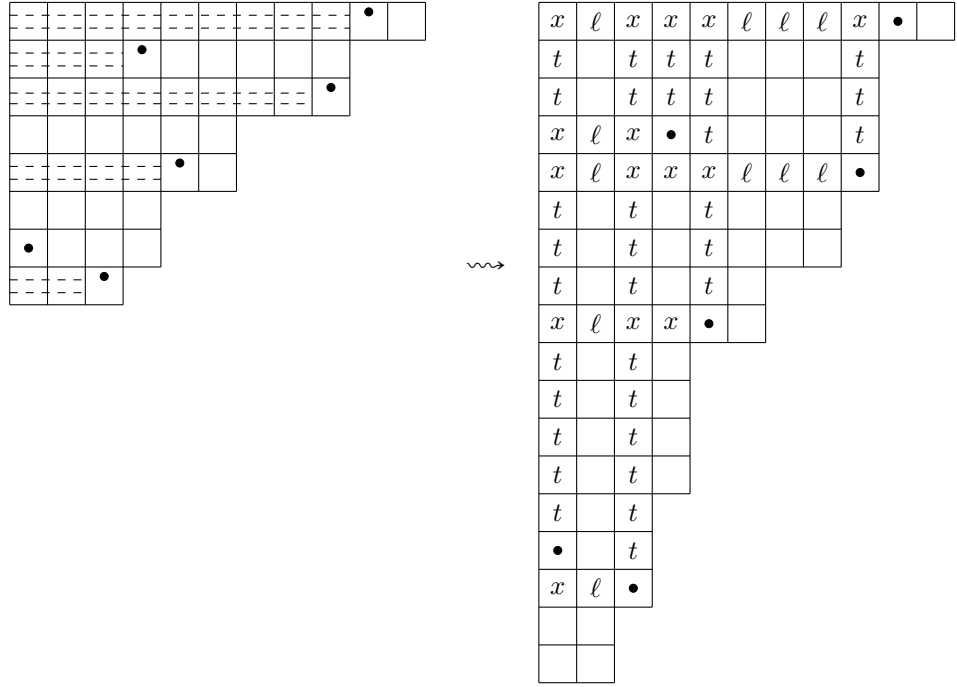
\begin{figure}[h]
		\centering
		\begin{tikzpicture}
			
			\draw (-0.5,4.5) -- (5,4.5);
			\draw (-0.5,4) -- (5,4);
			\draw (-0.5,3.5) -- (4,3.5);
			\draw (-0.5,3) -- (4,3);
			\draw (-0.5,2.5) -- (2.5,2.5);
			\draw (-0.5,2) -- (2.5,2);
			\draw (-0.5,1.5) -- (1.5,1.5);
			\draw (-0.5,1) -- (1.5,1);
			\draw (-0.5,0.5) -- (1,0.5);
			
			\draw[dashed] (-0.5,4.33) -- (4,4.33);
			\draw[dashed] (-0.5,4.16) -- (4,4.16);
			
			\draw[dashed] (-0.5,3.83) -- (1,3.83);
			\draw[dashed] (-0.5,3.66) -- (1,3.66);
			
			\draw[dashed] (-0.5,3.3) -- (3.5,3.3);
			\draw[dashed] (-0.5,3.16) -- (3.5,3.16);
			
			\draw[dashed] (-0.5,2.33) -- (1.5,2.33);
			\draw[dashed] (-0.5,2.16) -- (1.5,2.16);
			
			\draw[dashed] (-0.5,0.83) -- (0.5,0.83);
			\draw[dashed] (-0.5,0.66) -- (0.5,0.66);
			
			\draw (-0.5,0.5) -- (-0.5,4.5);
			\draw (0,0.5) -- (0,4.5);
			\draw (0.5,0.5) -- (0.5,4.5); 
			\draw (1,0.5) -- (1,4.5); 
			\draw (1.5,1) -- (1.5,4.5);
			\draw (2,2) -- (2,4.5);
			\draw (2.5,2) -- (2.5,4.5);
			\draw (3,3) -- (3,4.5);
			\draw (3.5,3) -- (3.5,4.5);
			\draw (4,3) -- (4,4.5);
			\draw (4.5,4) -- (4.5,4.5);
			\draw (5,4) -- (5,4.5);
			
			\filldraw (4.25,4.375) circle (1.5pt);
			\filldraw (1.25,3.875) circle (1.5pt);
			\filldraw (3.75,3.375) circle (1.5pt);
			\filldraw (1.75,2.375) circle (1.5pt);
			\filldraw (0.75,0.875) circle (1.5pt);
			\filldraw (-0.25,1.25) circle (1.5pt);
			
			\node at (5.8,1) {$\rightsquigarrow $};

			\draw (6.5,4.5) -- (12,4.5);
			\draw (6.5,4) -- (12,4);
			\draw (6.5,3.5) -- (11,3.5);
			\draw (6.5,3) -- (11,3);
			\draw (6.5,2.5) -- (11,2.5);
			\draw (6.5,2) -- (11,2);
			\draw (6.5,1.5) -- (10.5,1.5);
			\draw (6.5,1) -- (10.5,1);
			\draw (6.5,0.5) -- (9.5,0.5);
			\draw (6.5,0) -- (9.5,0);
			\draw (6.5,-0.5) -- (8.5,-0.5);
			\draw (6.5,-1) -- (8.5,-1);
			\draw (6.5,-1.5) -- (8.5,-1.5);
			\draw (6.5,-2) -- (8.5,-2);
			\draw (6.5,-2.5) -- (8,-2.5);
			\draw (6.5,-3) -- (8,-3);
			\draw (6.5,-3.5) -- (8,-3.5);
			\draw (6.5,-4) -- (7.5,-4);
			\draw (6.5,-4.5) -- (7.5,-4.5);
			
			\draw (6.5,-4.5) -- (6.5,4.5);
			\draw (7,-4.5) -- (7,4.5);
			\draw (7.5,-4.5) -- (7.5,4.5);
			\draw (8,-3.5) -- (8,4.5);
			\draw (8.5,-2) -- (8.5,4.5);
			\draw (9,0) -- (9,4.5);
			\draw (9.5,0) -- (9.5,4.5);
			\draw (10,1) -- (10,4.5);
			\draw (10.5,1) -- (10.5,4.5);
			\draw (11,2) -- (11,4.5);
			\draw (11.5,4) -- (11.5,4.5);
			\draw (12,4) -- (12,4.5);

			\filldraw (11.25,4.25) circle (1.5pt);
			\filldraw (10.75,2.25) circle (1.5pt);
			\filldraw (8.75,0.25) circle (1.5pt);
			\filldraw (8.25,2.75) circle (1.5pt);
			\filldraw (7.75,-3.25) circle (1.5pt);
			\filldraw (6.75,-2.75) circle (1.5pt);

			\node at (6.75,4.25) {$x$};
			\node at (6.75,3.75) {$t$};
			\node at (6.75,3.25) {$t$};
			\node at (6.75,2.75) {$x$};
			\node at (6.75,2.25) {$x$};
			\node at (6.75,1.75) {$t$};
			\node at (6.75,1.25) {$t$};
			\node at (6.75,0.75) {$t$};
			\node at (6.75,0.25) {$x$};
			\node at (6.75,-0.25) {$t$};
			\node at (6.75,-0.75) {$t$};
			\node at (6.75,-1.25) {$t$};
			\node at (6.75,-1.75) {$t$};
			\node at (6.75,-2.25) {$t$};
			\node at (6.75,-3.25) {$x$};
			
			\node at (7.25,4.25) {$\ell$};
			\node at (7.25,2.75) {$\ell$};
			\node at (7.25,2.25) {$\ell$};
			\node at (7.25,0.25) {$\ell$};
			\node at (7.25,-3.25) {$\ell$};
			
			\node at (7.75,4.25) {$x$};
			\node at (7.75,3.75) {$t$};
			\node at (7.75,3.25) {$t$};			
			\node at (7.75,2.75) {$x$};
			\node at (7.75,2.25) {$x$};
			\node at (7.75,1.75) {$t$};
			\node at (7.75,1.25) {$t$};
			\node at (7.75,0.75) {$t$};
			\node at (7.75,0.25) {$x$};
			\node at (7.75,-0.25) {$t$};
			\node at (7.75,-0.75) {$t$};
			\node at (7.75,-1.25) {$t$};
			\node at (7.75,-1.75) {$t$};
			\node at (7.75,-2.25) {$t$};
			\node at (7.75,-2.75) {$t$};

			\node at (8.25,4.25) {$x$};
			\node at (8.25,3.75) {$t$};
			\node at (8.25,3.25) {$t$};
			\node at (8.25,2.25) {$x$};
			\node at (8.25,0.25) {$x$};

			\node at (8.75,4.25) {$x$};
			\node at (8.75,3.75) {$t$};
			\node at (8.75,3.25) {$t$};
			\node at (8.75,2.75) {$t$};
			\node at (8.75,2.25) {$x$};
			\node at (8.75,1.75) {$t$};
			\node at (8.75,1.25) {$t$};
			\node at (8.75,0.75) {$t$};

			\node at (9.25,4.25) {$\ell$};
			\node at (9.25,2.25) {$\ell$};
			
			\node at (9.75,4.25) {$\ell$};
			\node at (9.75,2.25) {$\ell$};
			
			\node at (10.25,4.25) {$\ell$};
			\node at (10.25,2.25) {$\ell$};
			
			\node at (10.75,4.25) {$x$};
			\node at (10.75,3.75) {$t$};
			\node at (10.75,3.25) {$t$};
			\node at (10.75,2.75) {$t$};
			
		\end{tikzpicture}
		\caption{The Ferrers board $B_{88875533311}$ with a $6$-rook placement for $s=2$ (left) and the equivalent representation with boxes marked by their class (right).}\label{FigFileSRook}
	\end{figure}  
\end{example}
In general, we have the following result \cite{TLM2}.
\begin{theorem}\label{TheoWordNOS}
	Let $\mathrm{w}_{{\bf m}, {\bf n}}=Y^{n_r}X^{m_r}\cdots Y^{n_1}X^{m_1}$ be a word in the letters $X$ and $Y$ satisfying the commutation relations \eqref{pqdefgenWeyl} of the $(p,q)$-deformed generalized Weyl algebra $A_{s;h|p,q}$. Then one has the normal ordering result
	\begin{equation}\label{pqrookNOS}
		\mathrm{w}_{{\bf m}, {\bf n}}=\sum_{k=0}^{|{\bf m}|-m_1} \mathscr{R}_{s,h; p,q}[B_{\mathrm{w}_{{\bf m}, {\bf n}}}, k] \, Y^{|{\bf n}|+(s-1)k}Z_p^{k}X^{|{\bf m}|-k},
	\end{equation}
	where $B_{\mathrm{w}_{{\bf m}, {\bf n}}}$ is the Ferrers board associated to $\mathrm{w}_{{\bf m}, {\bf n}}$ and $\mathscr{R}_{s,h; p,q}[B_{\mathrm{w}_{{\bf m}, {\bf n}}}, k]$ is given by \eqref{PQRookWeS}.
\end{theorem}

We conclude this section by introducing some additional notations from \cite{TLM1} to simplify the resulting formulas. For ${\bf m}:=(m_1, \ldots, m_r)\in\mathbb{N}^r_0$ and ${\bf n}:=(n_1, \ldots, n_r)\in\mathbb{N}^r_0$, we abbreviate $({\bf 0},{\bf m})_r$ to be the set of all ${\bf k}:=(k_1,\ldots,k_r) \in \mathbb{N}_0^r$ such that $0 \leq k_j \leq m_j $ for $j=1,\ldots,r$. Furthermore, let us introduce
\begin{align*}
	{\mathcal Q}_{r|s}({\bf m},{\bf n},{\bf k}):=&\sum_{j=1}^r (m_j-k_j)\left\{\sum_{l=1}^{j}n_l+(s-1)\sum_{l=1}^{j-1}k_l\right\},\\
	{\mathcal P}_r({\bf m},{\bf k}):=& \sum_{j=1}^{r-1}(m_{j+1}-k_{j+1})\sum_{l=1}^{j}k_l, \\
	{\mathcal A}_{r|s;p,q}({\bf n},{\bf k}):=& \prod_{j=1}^r \prod_{\ell =0}^{k_j-1}\left[\sum_{l=1}^{j}n_l+(s-1)\sum_{l=1}^{j-1}k_l+(s-1) \ell \right]_{p, q}.
\end{align*}

The following theorem from \cite{TLM1} plays a crucial rule in our study of binomial formula.
	\begin{theorem}\label{thmword} Let $X, Y$ and $Z_p$ be variables satisfying the commutation relations \eqref{pqdefgenWeyl} of the $(p, q)$-deformed generalized Weyl algebra $A_{s;h|p,q}$. Then the normal ordered form of an arbitrary word $\omega=X^{m_r}Y^{n_r}\cdots X^{m_1}Y^{n_1}$, $m_k,n_k \in \mathbb{N}_0$,  is given by
		\begin{equation}\label{Main}
			\omega =\sum_{{\bf k}  \in ({\bf 0},{\bf m})_r} \bigg\{ h^{|{\bf k}|} p^{{\mathcal P}_r({\bf m},{\bf k})}  q^{{\mathcal Q}_{r|s}({\bf m},{\bf n},{\bf k})}  {\bf m\choose \bf k}_{p, q^{s-1}}  {\mathcal A}_{r|s;p,q}({\bf n},{\bf k}) \bigg\} Y^{|{\bf n}|+(s-1)|{\bf k}|}Z_{p}^{|{\bf k}|}X^{|{\bf m}|-|{\bf k}|}.
		\end{equation}
	\end{theorem}

Note that the expressions $\mathrm{w}_{{\bf m}, {\bf n}}=Y^{n_r}X^{m_r}\cdots Y^{n_1}X^{m_1}$ used in Theorem~\ref{TheoWordNOS} and $\omega=X^{m_r}Y^{n_r}\cdots X^{m_1}Y^{n_1}$ used in Theorem~\ref{thmword} are equivalent since in both expressions the value 0 is allowed for the exponents. Furthermore, if $\omega=X^{m_r}Y^{n_r}\cdots X^{m_1}Y^{n_1}$ is given with associated Ferrers board $B_{\omega}$ and $(p,q)$-deformed $s$-rook numbers $\mathscr{R}_{s,h; p,q}[B_{\omega}, k]$, then for all words $\omega (n,m):=Y^n\omega X^m$, $m,n \in \mathbb{N}_0$, one has $\mathscr{R}_{s,h; p,q}[B_{\omega (n,m)}, k]=\mathscr{R}_{s,h; p,q}[B_{\omega}, k]$. 
	
\section{The binomial formula $(X+Y)^m$\label{bfm}}
In this section, we consider the binomial $(X+Y)^m$ where $X$ and $Y$ satisfy \eqref{pqdefgenWeyl}. We first derive some general results before we consider the special cases $s=0,1,2,3$ in more detail. For more on the binomial formula in noncommuting variables see \cite{TMMS2011} or the review \cite{MS2021} and the references given therein. Another recent approach can be found in \cite{HJYZ2023}.

\subsection{General results}
For this, we first expand the binomial into a sum over all binary words over the alphabet $\{X,Y\}$ of length $m$, and then we use for each resulting word Theorem~\ref{thmword}. To start, let us recall the following result \cite[Proposition 3.3]{TMMS2018}.
\begin{theorem}\label{thr1}
There exists a bijection between the Young diagrams in the set $\mathcal{I}_m$  and the binary words over the alphabet $\{X,Y\}$ of length $m$.
\end{theorem}
We briefly recall the construction of this bijection. Let $\pi=\pi_1\pi_2\cdots\pi_m$ be any word over the alphabet $\{X,Y\}$ of length $m$. Let $\ell$ be the number of occurrences of the letter $X$ in $\pi$. Then the number of occurrences of the letter $Y$ in $\pi$ is given by ($m-\ell$). The numbers $\lambda_j$, for $1\leq j \leq \ell$, are defined by
\begin{align}\label{eqbij1}
\pi=Y^{m-\ell-\lambda_1}XY^{\lambda_1-\lambda_2}\cdots XY^{\lambda_{\ell-1}-\lambda_\ell}XY^{\lambda_\ell}.
\end{align}
By construction, $\lambda_\ell\geq 0$, $\lambda_j \geq \lambda_{j+1}$ (for $j=1,\ldots , \ell-1$), and $\lambda_1 \leq (m-\ell$), thus $\lambda=(\lambda_1,\lambda_2,\ldots,\lambda_\ell)\equiv \lambda_1\lambda_2\cdots\lambda_\ell$ is a partition with $Y_\lambda\in\mathcal{I}_{m-\ell,\ell}$. On the other hand, if $Y_\lambda$ is a Young diagram in $\mathcal{I}_{m-\ell,\ell}\subset\mathcal{I}_m$, then we define the word $\pi=\pi_1\pi_2\cdots\pi_m$ by \eqref{eqbij1}. Thus, $\pi$ is a binary word over the alphabet $\{X,Y\}$ of length $m$ with exactly $\ell$ letters of $X$ and exactly ($m-\ell$) letters of $Y$. Hence, we can map the set of all binary words over the alphabet $\{X,Y\}$ of length $m$ with exactly $\ell$ letters of $X$ and exactly ($m-\ell$) letters of $Y$ bijectively to the set $\mathcal{I}_{m-\ell,\ell}$. Clearly, this is a bijection between the set of all binary words over the alphabet $\{X,Y\}$ of length $m$ and the set $\mathcal{I}_m$.

\begin{example} In Figure~\ref{fig1}, from top to bottom and left to right, the words $\pi$ that represent the Young diagrams are given by $XXXX$, $YXXX$, $XYXX$, $XXYX$, $XXXY$, $YYXX$, $YXYX$, $YXXY$, $XYYX$, $XYXY$, $XXYY$, $YYYX$, $YYXY$, $YXYY$, $XYYY$, $YYYY$.
\end{example}

Now, by expanding the binomial $(X+Y)^m$, we get a sum over all binary words of length $m$ in $X$ and $Y$. Thus, by Theorem~\ref{thr1} and, in particular \eqref{eqbij1}, we can write
\begin{equation}\label{binomgen}
(X+Y)^m=\sum_{\ell=0}^m \sum_{\lambda \in \mathcal{I}_{m-\ell,\ell}}Y^{m-\ell-\lambda_1}XY^{\lambda_1-\lambda_2}\cdots XY^{\lambda_{\ell-1}-\lambda_\ell}XY^{\lambda_\ell}.
\end{equation}
Each summand is a word in $X$ and $Y$ of the form \eqref{wordxy} considered in Theorem~\ref{thmword}. Let us write, for $\lambda \in \mathcal{I}_{m-\ell,\ell}$,
$$
\omega_{\lambda}=XY^{\lambda_1-\lambda_2}\cdots XY^{\lambda_{\ell-1}-\lambda_\ell}XY^{\lambda_\ell}.
$$
Then we can write \eqref{binomgen} in the equivalent brief form
\begin{equation}\label{binomgen2}
(X+Y)^m=\sum_{\ell=0}^m \sum_{\lambda \in \mathcal{I}_{m-\ell,\ell}}Y^{m-\ell-\lambda_1}\omega_{\lambda},
\end{equation}
and it is clear that the main work consists in normal ordering $\omega_{\lambda}$. Note that the words $\omega_{\lambda}$ have a particular structure when compared to the general case. Here all $m_j=1$, for $j=1,\ldots,\ell$, thus $|{\bf m}|=|(1,1,\ldots,1)|=\ell$. Similarly, summing all exponents of the variables $Y$ yields $|{\bf n}_{\lambda}|=|(\lambda_1-\lambda_2,\lambda_2-\lambda_3,\ldots,\lambda_{\ell-1}-\lambda_{\ell},\lambda_{\ell} )|=\lambda_1$. Thus, in the summands -- see \eqref{Main} -- we get factors of
$$
Y^{|{\bf n}_{\lambda}|+(s-1)|{\bf k}|}Z_{p}^{|{\bf k}|}X^{|{\bf m}|-|{\bf k}|} \rightsquigarrow Y^{\lambda_1+(s-1)|{\bf k}|}Z_{p}^{|{\bf k}|}X^{\ell-|{\bf k}|}.
$$
Since here $m_j=1$, for $j=1,\ldots, \ell$ (i.e., ${\bf m}={\bf 1}$), we write more concretely ${\bf k} \in \{0,1\}^{\ell}$ instead of ${\bf k} \in ({\bf 0},{\bf 1})_{\ell}$ from above and obtain
\begin{equation}\label{Omega}
\omega_{\lambda}=\sum_{{\bf k} \in \{0,1\}^{\ell}} \bigg\{ h^{|{\bf k}|} p^{{\mathcal P}_{\ell}({\bf 1},{\bf k})}  q^{{\mathcal Q}_{\ell|s}({\bf 1},{\bf n}_{\lambda},{\bf k})}  {\bf 1\choose \bf k}_{p, q^{s-1}}  {\mathcal A}_{\ell|s;p,q}({\bf n}_{\lambda},{\bf k}) \bigg\} Y^{\lambda_1+(s-1)|{\bf k}|}Z_{p}^{|{\bf k}|}X^{\ell-|{\bf k}|}.
\end{equation}
As a first observation, recall that, for all ${\bf k} \in \{0,1\}^{\ell}$, one has ${\bf 1\choose \bf k}_{p, q^{s-1}} =1$. For the case at hand, we also introduce
\begin{align*}
{\mathcal Q}_{r|s}^{\ast}({\bf n},{\bf k}):=&{\mathcal Q}_{r|s}({\bf 1},{\bf n},{\bf k})=\sum_{j=1}^r (1-k_j)\left\{\sum_{l=1}^{j}n_l+(s-1)\sum_{l=1}^{j-1}k_l\right\},\\
{\mathcal P}_r^{\ast}({\bf k}):=&{\mathcal P}_r({\bf 1},{\bf k}) =\sum_{j=1}^{r-1}(1-k_{j+1})\sum_{l=1}^{j}k_l.
\end{align*}
We can now formulate the main result of this section.
\begin{theorem}\label{thmbinomial} Let $X, Y$ and $Z_p$ be variables satisfying the commutation relations \eqref{pqdefgenWeyl} of the $(p, q)$-deformed generalized Weyl algebra $A_{s;h|p,q}$. Then the binomial $(X+Y)^m$, $m \in \mathbb{N}$, can be written in normal ordered form as
\begin{equation}\label{BinomMain}
(X+Y)^m=\sum_{\ell=0}^m \sum_{k=0}^{\ell} h^{k} {\mathcal J}_{s,p,q}(m,\ell,k) \,Y^{m-\ell+(s-1)k}Z_{p}^{k}X^{\ell-k},
\end{equation}
where ${\mathcal J}_{s,p,q}(m,0,k)=1$ and, for $\ell \neq 0$, 
\begin{equation}\label{eqBinoMain}
{\mathcal J}_{s,p,q}(m,\ell,k)=\sum\limits_{{\bf k}\in \{0,1\}^{\ell} \atop |{\bf k}|=k}  p^{{\mathcal P}_{\ell}^{\ast}({\bf k})}\sum\limits_{\lambda \in \mathcal{I}_{m-\ell,\ell}}  q^{{\mathcal Q}_{\ell|s}^{\ast}({\bf n}_{\lambda},{\bf k})} {\mathcal A}_{\ell|s;p,q}({\bf n}_{\lambda},{\bf k}),
\end{equation}
where ${\bf n}_{\lambda}=(\lambda_1-\lambda_2,\lambda_2-\lambda_3,\ldots,\lambda_{\ell-1}-\lambda_{\ell},\lambda_{\ell} )$ for $\lambda \in \mathcal{I}_{m-\ell,\ell}$.
\end{theorem}
\begin{proof}
Inserting \eqref{Omega} into \eqref{binomgen2}, one obtains
$$
(X+Y)^m=\sum_{\ell=0}^m \sum_{\lambda \in \mathcal{I}_{m-\ell,\ell}}\sum_{{\bf k}\in \{0,1\}^{\ell}} \bigg\{ h^{|{\bf k}|} p^{{\mathcal P}_{\ell}^{\ast}({\bf k})}  q^{{\mathcal Q}_{\ell|s}^{\ast}({\bf n}_{\lambda},{\bf k})} {\mathcal A}_{\ell|s;p,q}({\bf n}_{\lambda},{\bf k}) \bigg\} \,Y^{m-\ell+(s-1)|{\bf k}|}Z_{p}^{|{\bf k}|}X^{\ell-|{\bf k}|},
$$
where we have used ${\bf 1\choose \bf k}_{p, q^{s-1}} =1$ and the definitions of the functions ${\mathcal P}_r^{\ast}({\bf k})$ and ${\mathcal Q}_{r|s}^{\ast}({\bf n},{\bf k})$ above. Changing the order of summation, we have
$$
(X+Y)^m=\sum_{\ell=0}^m \sum_{{\bf k}\in \{0,1\}^{\ell}}h^{|{\bf k}|} p^{{\mathcal P}_{\ell}^{\ast}({\bf k})}\bigg\{  \sum_{\lambda \in \mathcal{I}_{m-\ell,\ell}}  q^{{\mathcal Q}_{\ell|s}^{\ast}({\bf n}_{\lambda},{\bf k})} {\mathcal A}_{\ell|s;p,q}({\bf n}_{\lambda},{\bf k}) \bigg\} \,Y^{m-\ell+(s-1)|{\bf k}|}Z_{p}^{|{\bf k}|}X^{\ell-|{\bf k}|}.
$$
Splitting the sum $\sum_{{\bf k}\in \{0,1\}^{\ell}}=\sum_{k=0}^{\ell}\sum_{{\bf k}\in \{0,1\}^{\ell}, |{\bf k}|=k}$, and applying the definition of ${\mathcal J}_{s,p,q}(m,\ell,k)$, this yields the assertion.
\end{proof}

\begin{example} Let us consider $m=1$ and $m=2$ and compare \eqref{BinomMain} to the explicit computation using the commutation relations \eqref{pqdefgenWeyl}. Let us start with $m=1$. Clearly, in this case no commutation relation is used and it remains to check that the right-hand side of \eqref{BinomMain} reduces to $X+Y$. In fact,
$$
{\mathcal J}_{s,p,q}(1,0,0)=1, \,\,\, {\mathcal J}_{s,p,q}(1,1,0)=1, \,\,\, {\mathcal J}_{s,p,q}(1,1,1)=0.
$$
It follows that
\begin{align*}
\sum_{\ell =0}^{1}\sum_{k=0}^{\ell}h^{k}{\mathcal J}_{s,p,q}(1,\ell,k)Y^{1-\ell+(s-1)k}Z_{p}^{k}X^{\ell-k}&={\mathcal J}_{s,p,q}(1,0,0)Y+{\mathcal J}_{s,p,q}(1,1,0)X+h{\mathcal J}_{s,p,q}(1,1,1)Y^{s-1}Z_{p}\\
&=Y+X,
\end{align*}
as expected. Now, for the case $m=2$ and using \eqref{eqBinoMain}, we have
$$
{\mathcal J}_{s,p,q}(2,2,1)=0, \,\,\,{\mathcal J}_{s,p,q}(2,2,2)=0, \,\,\, {\mathcal J}_{s,p,q}(2,2,0)=1, \,\,\,{\mathcal J}_{s,p,q}(2,1,1)=1, \,\,\, {\mathcal J}_{s,p,q}(2,1,0)=1+q.
$$
The right-hand side of \eqref{BinomMain} is then given by
\begin{align*}
\sum_{\ell =0}^{2}\sum_{k=0}^{\ell}h^k{\mathcal J}_{s,p,q}(2,\ell,k)Y^{2-\ell+(s-1)k}Z_{p}^{k}X^{\ell-k}=&{\mathcal J}_{s,p,q}(2,0,0)Y^2+{\mathcal J}_{s,p,q}(2,1,0)YX+h{\mathcal J}_{s,p,q}(2,1,1)Y^{s}Z_{p}\\
&\hspace{-2.5cm}+{\mathcal J}_{s,p,q}(2,2,0)X^2+h{\mathcal J}_{s,p,q}(2,2,1)Y^{s-1}Z_{p}X+h^2{\mathcal J}_{s,p,q}(2,2,2)Y^{2(s-1)}Z_{p}^{2}\\
=&Y^2+(1+q)YX+hY^sZ_{p}+X^2.
\end{align*}
Using \eqref{pqcr1}, this equals $X^2+XY+YX+X^2$ -- which is the left-hand side of \eqref{BinomMain} for $m=2$. 
\end{example}
By changing the order of summation, we can write \eqref{eqBinoMain} equivalently as
\begin{equation}\label{eqBinoMain2}
{\mathcal J}_{s,p,q}(m,\ell,k)=\sum\limits_{\lambda \in \mathcal{I}_{m-\ell,\ell}}\left\{ \sum\limits_{{\bf k}\in \{0,1\}^{\ell} \atop |{\bf k}|=k}  p^{{\mathcal P}_{\ell}^{\ast}({\bf k})}  q^{{\mathcal Q}_{\ell|s}^{\ast}({\bf n}_{\lambda},{\bf k})} {\mathcal A}_{\ell|s;p,q}({\bf n}_{\lambda},{\bf k})\right\}.
\end{equation}
We will consider the case $p=q=1$ for different values of $s$, so we mention the case of arbitrary $s$ explicitly.
\begin{corollary} Let $p=q=1$. Then one has 
\begin{equation}\label{SNormalCo}
{\mathcal J}_{s,1,1}(m,\ell,k)=\sum\limits_{\lambda \in \mathcal{I}_{m-\ell,\ell}} \sum\limits_{{\bf k}\in \{0,1\}^{\ell} \atop |{\bf k}|=k}\prod_{j=1}^{\ell}\left(\lambda_{\ell - j+1}+(s-1)(k_1+\cdots + k_{j-1}))\right)^{k_j}.
\end{equation}
\end{corollary}
\begin{proof}
From \eqref{eqBinoMain2}, we find for $p=q=1$ that
\begin{equation}\label{CooExpand}
{\mathcal J}_{s,1,1}(m,\ell,k)=\sum\limits_{\lambda \in \mathcal{I}_{m-\ell,\ell}}\sum\limits_{{\bf k}\in \{0,1\}^{\ell} \atop |{\bf k}|=k} {\mathcal A}_{\ell|s;1,1}({\bf n}_{\lambda},{\bf k}).
\end{equation}
Here one has ${\bf n}_{\lambda}=(\lambda_1-\lambda_2,\lambda_2-\lambda_3,\ldots,\lambda_{\ell-1}-\lambda_{\ell},\lambda_{\ell} )$ such that $\sum_{l=1}^{j}n_l=\lambda_{\ell - j+1}$, yielding
$$
{\mathcal A}_{\ell|s;1,1}({\bf n}_{\lambda},{\bf k})=\prod_{j=1}^{\ell} \prod_{i =0}^{k_j-1}\left(\lambda_{\ell - j+1}+(s-1)\sum_{l=1}^{j-1}k_l+(s-1)i \right).
$$
Observe that here $k_j \in \{0,1\}$ for all $1\leq j \leq \ell$, so we can write the inner product equivalently as 
$$
 \prod_{i =0}^{k_j-1}\left(\lambda_{\ell - j+1}+(s-1)\sum_{l=1}^{j-1}k_l- i \right)=\left(\lambda_{\ell - j+1}+(s-1)(k_1+\cdots + k_{j-1}))\right)^{k_j}.
$$
Thus, we have found that
\begin{equation}\label{StructureA}
{\mathcal A}_{\ell|s;1,1}({\bf n}_{\lambda},{\bf k})=\prod_{j=1}^{\ell}\left(\lambda_{\ell - j+1}+(s-1)(k_1+\cdots + k_{j-1}))\right)^{k_j},
\end{equation}
showing the assertion.
\end{proof}

Recalling the definition of the $s$-rook numbers from Definition~\ref{SRook}, we can now give the sum over ${\mathcal A}_{\ell|s;1,1}({\bf n}_{\lambda},{\bf k})$ a nice interpretation.
\begin{proposition} Let $p=q=1$. Then one has from \eqref{StructureA} that
\begin{equation}\label{ConnRook}
\sum\limits_{{\bf k}\in \{0,1\}^{\ell} \atop |{\bf k}|=k}{\mathcal A}_{\ell|s;1,1}({\bf n}_{\lambda},{\bf k})=r_k^{(s)}(B_{\lambda}),
\end{equation}
where $B_{\lambda}$ is the Ferrers board which for $\lambda \in \mathcal{I}_{m-\ell,\ell}$ has columns of length $\lambda_1,\lambda_2,\ldots, \lambda_{\ell}$. In particular, for $s=0$, the right-hand side equals the conventional rook number, $r_k(B_{\lambda})$, while, for $s=1$, it equals the file number, $f_k(B_{\lambda})$. 
\end{proposition}
\begin{proof} Let us denote by $B_{\lambda}$ the Ferrers board which for $\lambda \in \mathcal{I}_{m-\ell,\ell}$ has columns of length $\lambda_1,\lambda_2,\ldots, \lambda_{\ell}$. We consider first $s=0,1$ explicitly, before turning to the general case. For $s=1$, one has 
$$
\sum\limits_{{\bf k}\in \{0,1\}^{\ell} \atop |{\bf k}|=k}{\mathcal A}_{\ell|s;1,1}({\bf n}_{\lambda},{\bf k})=\sum\limits_{{\bf k}\in \{0,1\}^{\ell} \atop |{\bf k}|=k} \lambda_{1}^{k_{\ell}} \cdots \lambda_{\ell -1}^{k_2} \lambda_{\ell}^{k_1}. 
$$
A given summand $ \lambda_{1}^{k_{\ell}} \cdots\lambda_{\ell -1}^{k_2} \lambda_{\ell}^{k_1} $ counts the number of possible placements of $k$ rooks in $k$ chosen columns of the board (those where $k_j=1$) where at most one rook is placed in each column. In other words, it counts all possible $k$-file placements in the chosen columns. Thus, the sum over all possible choices of $k$ columns then gives the $k$-file number of the board, i.e., 
$$
\sum\limits_{{\bf k}\in \{0,1\}^{\ell} \atop |{\bf k}|=k}\lambda_{1}^{k_{\ell}} \cdots \lambda_{\ell -1}^{k_2} \lambda_{\ell}^{k_1}=f_k(B_{\lambda}).
$$ 
Since $f_k(B_{\lambda})=r_k^{(1)}(B_{\lambda})$, this shows the assertion for $s=1$. Let us turn to $s=0$. From \eqref{StructureA} one has that the left-hand side of \eqref{ConnRook} is given by
$$
\sum\limits_{{\bf k}\in \{0,1\}^{\ell} \atop |{\bf k}|=k} \left(\lambda_{1}-(k_1+\cdots+k_{\ell-1}))\right)^{k_{\ell}} \cdots \left(\lambda_{\ell -2}-(k_1+k_2))\right)^{k_3} \left(\lambda_{\ell -1}-k_1)\right)^{k_2}\left(\lambda_{\ell}\right)^{k_1}.
$$
Similar to the case $s=1$ considered above, the sum counts all possible choices of $k$ columns in which a rook will be placed (again, at most one rook is placed in each column -- $k_j=1$ means to place a rook into the $(\ell-j+1)$-th column). One summand counts the possibilities to place the $k$ rooks into the $k$ chosen columns in such a fashion that no two rook are in the same row. For example, if $k_1=1$, i.e., one rook is placed in the last column where one has $\lambda_{\ell}$ cells to choose, then in all columns to the left there is one possibility less to place a rook. Thus, since no two rooks are in the same row or column (i.e., they are non-attacking) the summand counts the number of possible rook-placements of $k$ rooks into the chosen $k$ columns. In total, left-hand side of \eqref{ConnRook} counts the $k$-th rook number of $B_{\lambda}$,
$$
\sum\limits_{{\bf k}\in \{0,1\}^{\ell} \atop |{\bf k}|=k} \left(\lambda_{1}-(k_1+\cdots+k_{\ell-1}))\right)^{k_{\ell}} \cdots \left(\lambda_{\ell -2}-(k_1+k_2))\right)^{k_3} \left(\lambda_{\ell -1}-k_1)\right)^{k_2}\left(\lambda_{\ell}\right)^{k_1}=r_k(B_{\lambda}).
$$
Since $r_k(B_{\lambda})=r_k^{(0)}(B_{\lambda})$, this shows the assertion for $s=0$. Let us now assume $s\geq 2$. Then the left-hand side of \eqref{ConnRook} is given by
$$
\sum\limits_{{\bf k}\in \{0,1\}^{\ell} \atop |{\bf k}|=k} \left(\lambda_{1}+(s-1)(k_1+\cdots+k_{\ell-1}))\right)^{k_{\ell}} \cdots \left(\lambda_{\ell -2}+(s-1)(k_1+k_2))\right)^{k_3} \left(\lambda_{\ell -1}+(s-1)k_1)\right)^{k_2}\left(\lambda_{\ell}\right)^{k_1}.
$$
The interpretation is similar to above. If we start placing rooks from right to left, we let $r$ be the smallest index with $k_r=1$. Thus, we have $\lambda_{\ell-r+1}$ choices to place a rook into this column. Observe that we then have in all columns to the left of this row $(s-1)$ additional choices to place a rook, corresponding to, e.g., the factor $ \left(\lambda_{\ell-r}+s-1\right)^{k_{r+1}}$ for the column directly to the left of it. Thus, we can interpret this in such a fashion that in all columns to the left of the column where we placed the rook, the row of the rook was split into $(s+1)$ rows, giving $s$ new rows (of which only $(s-1)$ are used for a placement of a rook). But this is exactly the description of the $s$-row creation rule described before Definition~\ref{SRook}. Thus, the left-hand side of \eqref{ConnRook} counts the $k$-th $s$-rook number of $B_{\lambda}$,
$$
\sum\limits_{{\bf k}\in \{0,1\}^{\ell} \atop |{\bf k}|=k} \left(\lambda_{1}+(s-1)(k_1+\cdots+k_{\ell-1}))\right)^{k_{\ell}} \cdots  \left(\lambda_{\ell -1}+(s-1)k_1)\right)^{k_2}\left(\lambda_{\ell}\right)^{k_1}=r_k^{(s)}(B_{\lambda}),
$$
showing the assertion for $s\geq 2$.
\end{proof}

By combining \eqref{CooExpand} and \eqref{ConnRook}, we obtain the following result.
\begin{corollary}\label{BinomialRook} Let $p=q=1$. Then the normal ordering coefficient ${\mathcal J}_{s,1,1}(m,\ell,k)$ is given by
\begin{equation}
{\mathcal J}_{s,1,1}(m,\ell,k)= \sum\limits_{\lambda \in \mathcal{I}_{m-\ell,\ell}}r_k^{(s)}(B_{\lambda}).
\end{equation}
In particular, \eqref{BinomMain} can be written as 
\begin{equation}
(X+Y)^m=\sum_{\ell=0}^m \sum_{k=0}^{\ell}  \sum\limits_{\lambda \in \mathcal{I}_{m-\ell,\ell}} h^{k} \, r_k^{(s)}(B_{\lambda}) \,Y^{m-\ell+(s-1)k}X^{\ell-k}.
\end{equation}
\end{corollary}
By using a result of Celeste et al. \cite{CCG2017}, we easily obtain a $q$-analog to Corollary~\ref{BinomialRook}. 
\begin{corollary}\label{qBinomialRook} Let $p=1$. Then the normal ordering coefficient ${\mathcal J}_{s,1,q}(m,\ell,k)$ is given by
\begin{equation}\label{qBinomCoeff}
{\mathcal J}_{s,1,q}(m,\ell,k)= \sum\limits_{\lambda \in \mathcal{I}_{m-\ell,\ell}} R_{s, 1; q}[B_{\lambda}, k],
\end{equation}
where the $q$-deformed $s$-rook number $R_{s, 1; q}[B_{\lambda}, k]$ was introduced in \cite{CCG2017} by
\begin{equation}\label{RookCeleste}
	R_{s, 1; q}[B_{\lambda}, k] :=\sum\limits_{\phi\in\mathcal{R}_s(B_{\lambda}, k)} \omega_{s;q}(\phi)= \sum\limits_{\phi\in\mathcal{R}_s(B_{\lambda}, k)} q^{\# e-boxes}.
\end{equation}
In particular, \eqref{BinomMain} can be written as 
\begin{equation}\label{qBinomExpr}
(X+Y)^m=\sum_{\ell=0}^m \sum_{k=0}^{\ell}  \sum\limits_{\lambda \in \mathcal{I}_{m-\ell,\ell}} h^{k} \, R_{s, 1; q}[B_{\lambda}, k] \, Y^{m-\ell+(s-1)k}X^{\ell-k}.
\end{equation}
\end{corollary}
\begin{proof}
Recall that the binomial $(X+Y)^m$ can be expanded as in \eqref{binomgen}. Note that each summand in the sum on the right-hand side of \eqref{binomgen}, i.e.,  $Y^{m-\ell-\lambda_1}XY^{\lambda_1-\lambda_2}\cdots XY^{\lambda_{\ell-1}-\lambda_\ell}XY^{\lambda_\ell}X^0$, is a word of the form $\mathrm{w}_{{\bf m}, {\bf n}}=Y^{n_r}X^{m_r}\cdots Y^{n_1}X^{m_1}$, where $|{\bf m}|= \ell$ and $|{\bf n}|=m-\ell$. Thus, using the result 
\begin{equation}\label{qrookfb}
	\mathrm{w}_{{\bf m}, {\bf n}}=\sum_{j=0}^{|{\bf{m}}|-m_1} R_{s, h; q}[B_{\mathrm{w}_{{\bf m}, {\bf n}}}, j] \, Y^{|{\bf{n}}|-j(1-s)}X^{|{\bf{m}}|-j}
\end{equation}
of Celeste et al. \cite{CCG2017} (which is, in fact, the case $p=1$ of Theorem~\ref{TheoWordNOS}), we find
$$
(X+Y)^m=\sum_{\ell=0}^m \sum_{\lambda \in \mathcal{I}_{m-\ell,\ell}}\sum_{j=0}^{\ell } R_{s, h; q}[B_{\lambda}, j] \, Y^{m-\ell-j(1-s)}X^{\ell-j},
$$
where we used $m_1=0$ and that the board associated to $Y^{m-\ell-\lambda_1}XY^{\lambda_1-\lambda_2}\cdots XY^{\lambda_{\ell-1}-\lambda_\ell}XY^{\lambda_\ell}X^0$ is $B_{\lambda}$. Since $ R_{s, h; q}[B_{\lambda}, j]= h^j R_{s, 1; q}[B_{\lambda}, j]$, this shows \eqref{qBinomExpr}. Comparing \eqref{qBinomExpr} with \eqref{BinomMain} yields \eqref{qBinomCoeff}.
\end{proof}
Comparing \eqref{eqBinoMain2} for $p=1$ with \eqref{qBinomCoeff} yields the following $q$-analog to \eqref{ConnRook},
\begin{equation}
\sum\limits_{{\bf k}\in \{0,1\}^{\ell} \atop |{\bf k}|=k} q^{{\mathcal Q}_{\ell|s}^{\ast}({\bf n}_{\lambda},{\bf k})} {\mathcal A}_{\ell|s;1,q}({\bf n}_{\lambda},{\bf k})=R_{s, 1; q}[B_{\lambda}, k],
\end{equation}
where ${\mathcal A}_{\ell|s;1,q}({\bf n}_{\lambda},{\bf k})$ is given by the $q$-deformed version of \eqref{StructureA} and ${\mathcal Q}_{\ell|s}^{\ast}({\bf n}_{\lambda},{\bf k})$ can be written explicitly as
$$
{\mathcal Q}_{\ell|s}^{\ast}({\bf n}_{\lambda},{\bf k})=(1-k_1)\lambda_{\ell}+(1-k_2)(\lambda_{\ell-1}+(s-1)k_1)+\cdots + (1-k_{\ell})(\lambda_{1}+(s-1)(k_1+\cdots + k_{\ell-1})).
$$

Using the same argument as for the case $p=1$ above and using the normal ordering results of Theorem~\ref{TheoWordNOS} imply the following result.
\begin{corollary}\label{pqBinomialRook} Let $p\neq 1\neq q$. The normal ordering coefficient ${\mathcal J}_{s,p,q}(m,\ell,k)$ is given, for all $s \in \mathbb{N}_0$, by
\begin{equation}\label{pqBinomCoeff}
{\mathcal J}_{s,p,q}(m,\ell,k)= \sum\limits_{\lambda \in \mathcal{I}_{m-\ell,\ell}} \mathscr{R}_{s,1; p,q}[B_{\lambda},k],
\end{equation}
where $\mathscr{R}_{s,1; p,q}[B_{\lambda},k]$ is defined for $s>0$ in \eqref{PQRookWeS}, and where we let $\mathscr{R}_{0,1; p,q}[B_{\lambda},k]:=\mathscr{R}_{1; p,q}[B_{\lambda},k]$ defined in \eqref{PQRookWe}. In particular, \eqref{BinomMain} can be written, for all $s \in \mathbb{N}_0$, as 
\begin{equation}\label{pqBinomExpr}
(X+Y)^m=\sum_{\ell=0}^m \sum_{k=0}^{\ell}  \sum\limits_{\lambda \in \mathcal{I}_{m-\ell,\ell}} h^{k} \, \mathscr{R}_{s,1; p,q}[B_{\lambda},k] \, Y^{m-\ell+(s-1)k}Z_p^{k}X^{\ell-k}.
\end{equation}
\end{corollary}

\begin{remark}
We considered the binomial formula for variables satisfying $XY-qYX=hf(Y)Z_p$ where $f(Y)=Y^s$ is a monomial. It would be interesting to extend the considerations to the case of polynomials $f\in \mathbb{C}[X,Y]$. Even for $p=1$, only special cases have been considered. The case $f(X,Y)=X+Y$ has been studied thoroughly by many authors due to its combinatorial (permutation tableaux) and physical ramifications, see \cite[Section 8.4.3]{TMMS2016} or the review \cite{AWRBME2020} for some references. A quadratic polynomial, i.e., $f(X,Y)=aX^2+bY^2$, was considered by Rosengren \cite{HR2000} and, recently, physical application of the associated binomial formula have been found in models of statistical physics, see, e.g., \cite{GBMR2023,AP2013}.
 \end{remark}
 
\subsection{Viskov's result ($p=q=1$) and some consequences}\label{SectViskov}
Viskov \cite{OVV1996} described a procedure to determine the exponential of the sum of particular functions of two variables satisfying particular commutation relations, see also the recent discussion by Beauduin \cite{B2024}. As a special case of a more general result, Viskov showed the following result (using our notation).
\begin{proposition}[Viskov \cite{OVV1996}] \label{PBinVis}Let $X$ and $Y$ satisfy $XY-YX=\varphi(Y)$ where $\varphi$ is a given polynomial. Then
\begin{equation}\label{BinViskov}
\exp \left(t(X+Y)\right)=\exp \left( \int_0^t \hat{Y}(\tau) \, d\tau \right) \exp\left(tX \right),
\end{equation}
where $ \hat{Y}(t)$ is a solution of the Cauchy problem $\frac{d \hat{Y}}{dt}=\varphi(\hat{Y}), \hat{Y}|_{t=0}=Y$.
\end{proposition} 
The commutation relation assumed in Proposition~\ref{PBinVis} is exactly the case $p=q=1$ of \eqref{pqcr1}. Since $(X+Y)^n=n! [t^n] \exp \left(t(X+Y)\right)$ one obtains from Proposition~\ref{PBinVis} the following result. 
\begin{corollary} Let $X$ and $Y$ satisfy $XY-YX=\varphi(Y)$ where $\varphi$ is a given polynomial. Then
\begin{equation}\label{BinViskovCor}
(X+Y)^n=n! [t^n]  \left\{\exp \left( \int_0^t \hat{Y}(\tau) \, d\tau \right) \exp\left(tX \right) \right\},
\end{equation}
where $ \hat{Y}(t)$ is a solution of the Cauchy problem $\frac{d \hat{Y}}{dt}=\varphi(\hat{Y}), \hat{Y}|_{t=0}=Y$.
\end{corollary}
\begin{example}[Weyl algebra] Consider $\varphi(Y)=hI$, i.e., $X,Y$ satisfy the commutation relation $XY-YX=hI$ of the Weyl algebra. The Cauchy problem becomes $\frac{d \hat{Y}}{dt}=hI, \hat{Y}|_{t=0}=Y$, with solution $\hat{Y}(t)=Y+htI$. Thus, $\int_0^t \hat{Y}(\tau) \, d\tau =Yt+\frac{h}{2}It^2$, and \eqref{BinViskov} implies that (note that there is a small typo in \cite{OVV1996})
$$
\exp \left(t(X+Y)\right)=\exp\left(\frac{ht^2}{2}\right) \exp\left(tY\right) \exp\left(tX \right),
$$
which is a well-known result (sometimes called ``splitting formula''), see \cite[Section 6.1.5]{TMMS2016} for a discussion and references. Using \eqref{BinViskovCor}, and expanding the exponentials, yields
$$
(X+Y)^n=n! [t^n]  \left\{ \sum_{k=0}^{\infty}\sum_{r=0}^{\infty}\sum_{s=0}^{\infty}\frac{h^k}{2^k k! r! s !}Y^rX^{s}t^{2k+r+s}\right\}.
$$
Using $s=\ell-k$ and $r=n-\ell -k$, one has $2k+r+s=2k+(n-\ell -k)+(\ell-k)=n$, hence
\begin{equation}\label{Yamazaki}
(X+Y)^n =\sum_{\ell=0}^{n}\sum_{k=0}^{\ell}h^k \frac{ n!}{2^k k! (\ell -k)! (n-\ell -k)!}Y^{n-\ell -k}X^{\ell-k},
\end{equation}
which is a well-known identity in the Weyl algebra, see \cite[Section 6.1.1]{TMMS2016} for some history of this remarkable formula (and some close relatives). We will return to it from a different point of view in Section~\ref{SecWeylBin}. 
\end{example}
Let us specialize the commutation relation assumed in the above results to $XY-YX=hY^s$ for $s\in \mathbb{N}_0$, i.e., $\varphi(Y)=hY^s$, is a monomial. The case $s=0$ corresponds to the preceding example, and for the cases $s=1$ and $s=2$ the solution of the corresponding Cauchy problems can be found in \cite{OVV1996}, see also the discussion in \cite[Section 8.4.1]{TMMS2016}.
\begin{example}[Shift algebra] Let $s=1$, i.e., $X$ and $Y$ satisfy the relation $XY-YX=hY$ of the shift algebra. The solution of the Cauchy problem $\frac{d \hat{Y}}{dt}=h\hat{Y}(t), \hat{Y}|_{t=0}=Y$ is given by $\hat{Y}(t)=Ye^{ht}$, implying $ \int_0^t \hat{Y}(\tau) \, d\tau =\frac{1}{h}(e^{ht}-1)Y$. Thus, \eqref{BinViskov} becomes an identity due to Sack \cite{RAS1958},
\begin{equation}\label{Sack}
\exp \left(t(X+Y)\right)=\exp\left(\frac{1}{h}(e^{ht}-1)Y\right) \exp\left(tX \right).
\end{equation}
Similar as in the preceding example, one finds by expanding the exponentials 
\begin{align*}
(X+Y)^n=&n! [t^n]  \left\{ \sum_{k=0}^{\infty} \sum_{r=0}^{\infty} \frac{1}{k!}\frac{1}{r !} \left(\frac{1}{h}(e^{ht}-1)\right)^kY^k X^{r} t^{r} \right\}\\=&n! [t^n]  \left\{ \sum_{k=0}^{\infty} \sum_{r=0}^{\infty}\sum_{m=0}^{\infty}\frac{1}{r !}\frac{1}{m!} h^{m-k} S(m,k)Y^k X^{r} t^{r+m} \right\},
\end{align*}
where we used in the second equation the exponential generating function of the Stirling numbers of the second kind. Letting $r=n-m$, and using the Bell polynomials $B_m(z)=\sum_{k} S(m,k) z^{k}$, one finds
\begin{equation}\label{VisShift}
(X+Y)^n=\sum_{m=0}^{n} \binom{n}{m} h^{m} B_m(Y/h) X^{n-m}.
\end{equation}
This result and line of argument can be found in \cite{OVV1995} (note that there is a slight typo in \cite{OVV1995}).
\end{example}

\begin{example}[Jordan plane] Let $s=2$, i.e., $X$ and $Y$ satisfy the relation $XY-YX=hY^2$ of the Jordan plane. The solution of the Cauchy problem $\frac{d \hat{Y}}{dt}=h(\hat{Y}(t))^2, \hat{Y}|_{t=0}=Y$ is given by $\hat{Y}(t)=\frac{Y}{(1-hYt)}$, implying $ \int_0^t \hat{Y}(\tau) \, d\tau =(1-htY)^{-\frac{1}{h}}$. Thus, \eqref{BinViskov} becomes an identity due to Berry \cite{MVB1966},
\begin{equation}\label{Berry}
\exp \left(t(X+Y)\right)=(1-htY)^{-\frac{1}{h}} \exp\left(tX \right).
\end{equation}
Expanding the exponential and using Newton's binomial series, we find
\begin{align*}
(X+Y)^n=n! [t^n]  \left\{\sum_{k=0}^{\infty} \sum_{m=0}^{\infty} \binom{-\frac{1}{h}}{k} \frac{1}{m!} (-h)^k Y^k X^mt^{m+k}\right\}.
\end{align*}
Thus, if $m=n-k$, this yields
$$
(X+Y)^n= \sum_{k=0}^{n} \binom{-\frac{1}{h}}{k} \frac{n!}{(n-k)!} (-h)^k Y^k X^{n-k}.
$$
Using $\binom{-\frac{1}{h}}{k} \frac{n!}{(n-k)!}=\binom{n}{k}\frac{1}{(-h)^k}\prod_{j=1}^{k-1}(1+hj)$, we recover the binomial formula due to Benaoum \cite{HBB1998},
\begin{equation}\label{ViskovJordan}
(X+Y)^n= \sum_{k=0}^{n} \binom{n}{k}\prod_{j=1}^{k-1}(1+hj)Y^k X^{n-k}.
\end{equation}
Defining the {\em $h$-binomial coefficients}, $\binom{n}{k}_h:=\binom{n}{k}\prod_{j=1}^{k-1}(1+hj)$, it was mentioned in the Introduction. We will return to it in Section~\ref{SectBinomJordan} where also an interpretation and a $q$-deformation will be discussed.
\end{example}
Following the above examples, it is in principle clear how to proceed for the case $XY-YX=hY^s$ for $s\in \mathbb{N}$ with $s\geq 3$. According to \cite[Corollary 8.57]{TMMS2016}, one has in this case
\begin{equation}\label{VisGen}
\exp \left(t(X+Y)\right)=\exp \left\{\frac{1-\left(1-t(s-1)hY^{s-1} \right)^{\frac{s-2}{s-1}}}{h(s-2)Y^{s-2}} \right\} \exp\left(tX \right).
\end{equation}
An equivalent expression is given by Beauduin \cite[Theorem 6.1]{B2024}. At the same place, one can also find the resulting binomial formula as follows. 
\begin{theorem}\cite[Theorem 6.1]{B2024} Let $X$ and $Y$ satisfy $XY-YX=hY^s$ with $s\in \mathbb{N}$. Then one has
\begin{equation}\label{BinBeau}
(X+Y)^n=\sum_{k=0}^n \sum_{j=0}^k \binom{n}{k} \mathfrak{S}_{s-1;h}(k,j) Y^{(s-1)(k-j)+j}X^{n-k},
\end{equation}
with the generalized Stirling numbers $\mathfrak{S}_{s;h}(k,j|1,1)\equiv \mathfrak{S}_{s;h}(k,j)=h^{k-j}\mathfrak{S}_{s;1}(k,j)$ considered in \eqref{StirGen}, for $p=q=1$.
\end{theorem}
Let us specialize to $s=3$. From \eqref{VisGen}, we obtain 
\begin{equation}\label{VisS3}
\exp \left(t(X+Y)\right)=\exp \left\{\frac{1-\sqrt{1-2hY(tY)}}{hY} \right\} \exp\left(tX \right).
\end{equation}
The Bessel polynomials are defined, for $\ell \in \mathbb{N}_0$, by
$$
y_{\ell}(x):=\sum_{k=0}^{\ell} \frac{(\ell +k)!}{2^k k! (\ell -k)!}x^k,
$$
and they have the generating function 
$$
\sum_{n=0}^{\infty}\frac{t^n}{n!}y_{n-1}(x)=e^{\frac{1-\sqrt{1-2xt}}{x}}.
$$
Thus, we conclude that
$$
\exp \left(t(X+Y)\right)=\sum_{n=0}^{\infty}\frac{(Yt)^n}{n!}y_{n-1}(hY) e^{tX}.
$$
By expanding the exponentials and comparing coefficients, we recover the following binomial formula \cite[Theorem 2.3]{TMMS2013}.
\begin{proposition}\label{Bessel} Let $X$ and $Y$ satisfy the relation $XY-YX=-Y^3$. Then one has 
\begin{equation}\label{bf2823}
(X+Y)^n= \sum_{k=0}^n \binom{n}{k} Y^{k}y_{k-1}(-Y)X^{n-k},
\end{equation}
where $y_{\ell}(x)$ are Bessel polynomials defined above.
\end{proposition}
The equivalent version given in \cite[Theorem 2.3]{TMMS2013} can be obtained by letting $m = n-k$,
\begin{equation}\label{bf2823b}
(X+Y)^n= \sum_{m=0}^n \binom{n}{m} Y^{n-m}y_{n-m-1}(-Y)X^{m}.
\end{equation}
\begin{remark}\label{RemFrac}
To treat the Stirling numbers of order $\frac{1}{2}$, a generalized Weyl algebra $\mathcal{W}^{\frac{1}{2}}$ was introduced in \cite[Definition 4.1]{MS2024} with generators $U,V,W$ (representing abstractly $D,X^{\frac{1}{2}},X^{-\frac{1}{2}}$) where in particular $UW-WU=-\frac{1}{2}W^3$. Thus, \eqref{bf2823b} gives an expression for $(U+W)^n$ in $\mathcal{W}^{\frac{1}{2}}$.
\end{remark}
For arbitrary $s\in \mathbb{N}$, we can write \eqref{BinBeau} in an equivalent form similar to \eqref{bf2823b}. This is the straightforward generalization of Viskov's result \eqref{VisShift}. 
\begin{proposition} Let $X$ and $Y$ satisfy $XY-YX=hY^s$ with $s\in \mathbb{N}$. Then one has
\begin{equation}\label{ViskovS}
(X+Y)^n= \sum_{m=0}^n \binom{n}{m} h^{n-m}\, Y^{(n-m)(s-1)} \mathfrak{B}^{(s-1)}_{n-m}(Y^{2-s}/h) X^{m} ,
\end{equation}
where 
\begin{equation}\label{GBPpq1}
	\mathfrak{B}^{(s)}_{\ell}(x)\equiv \mathfrak{B}^{(s;1)}_{\ell|1,1}(x)=\sum_{k=0}^{\ell}\mathfrak{S}_{s;1}(\ell,k|1,1)x^k\equiv\sum_{k=0}^{\ell}\mathfrak{S}_{s;1}(\ell,k)x^k,
\end{equation} 
are generalized Bell polynomials defined in \cite{LO2021}, for $p=q=1$, and satisfy
\begin{equation}\label{IdGBPpq1}
\mathfrak{B}^{(s;h)}_{\ell|1,1}(x)=h^\ell\mathfrak{B}^{(s;1)}_{\ell|1,1}(x/h).	
\end{equation}
\end{proposition}
\begin{proof}
Let us introduce in \eqref{BinBeau} $m = n-k$, hence
$$
(X+Y)^n=\sum_{m=0}^n \binom{n}{m} Y^{(n-m)(s-1)}\left( \sum_{j=0}^{m}  \mathfrak{S}_{s-1;h}(n-m,j) Y^{-(s-2)j}\right) X^{m},
$$
where we used that $(s-1)(n-m-j)+j=(n-m)(s-1)-(s-2)j$. Observe that
$$
\sum_{j=0}^{m}  \mathfrak{S}_{s-1;h}(n-m,j) Y^{-(s-2)j}=\mathfrak{B}^{(s-1;h)}_{n-m|1,1}(Y^{2-s})=h^{n-m}\, \mathfrak{B}^{(s-1;1)}_{n-m|1,1}(Y^{2-s}/h),
$$
where we used in the first equation \eqref{GBPpq1} and in the second equation \eqref{IdGBPpq1}. Abbreviating $\mathfrak{B}^{(s-1)}_{n-m}(x)\equiv \mathfrak{B}^{(s-1;1)}_{n-m|1,1}(x)$, this shows the assertion.
\end{proof}
Note that choosing $s=1$ in \eqref{ViskovS} and switching to $k=n-m$, one recovers \eqref{VisShift} since $\mathfrak{B}^{(0)}_{k}(x) \equiv B_k(x)$. Choosing $s=2$ in \eqref{ViskovS} and switching to $k=n-m$, one finds
$$
(X+Y)^n= \sum_{k=0}^n \binom{n}{k} h^{k}\, \mathfrak{B}^{(1)}_{k}(h^{-1}) Y^{k} X^{n-k}.
$$ 
Recalling $\mathfrak{S}_{1;1}(k,j)=c(k,j)$, one has $h^{k}\, \mathfrak{B}^{(1)}_{k}(h^{-1})=\prod_{j=0}^{k-1}(1+hj)$, thereby recovering \eqref{ViskovJordan}. 

Furthermore, choosing $s=3$ and $h=-1$ in \eqref{ViskovS} yields upon comparison with \eqref{bf2823b} the identity $ (-Y)^{(n-m)} \mathfrak{B}^{(2)}_{n-m}(-Y^{-1})=y_{n-m-1}(-Y)$, or, 
$$ \mathfrak{B}^{(2)}_{\ell}(x)=y_{\ell -1}(x^{-1})x^{\ell}.$$

In the following sections, we consider \eqref{BinomMain} for $s \in \{0,1,2,3\}$, corresponding to $(p,q)$-deformations of well-known structures. Before we do this, we consider briefly the case $h=0$ corresponding to $q$-commuting variables, i.e., the quantum plane.

\subsection{The quantum plane}\label{BinomPS}
As first example, let us recover the conventional $q$-binomial formula \eqref{BFPS} due to Potter \cite{HSAP1950} and Schützenberger \cite{MPS1953}.
\begin{corollary}\label{Potter} Let $h=0$, i.e., one has $XY=qYX$. Then one has, for $m \in \mathbb{N}$, that
$$
(X+Y)^m=\sum_{\ell=0}^m  {m\choose \ell}_{q} \,Y^{m-\ell}X^{\ell}.
$$
\end{corollary}
\begin{proof}
From \eqref{BinomMain}, one finds for $h=0$ that only the summand $k=0$ remains, i.e.,
$$
(X+Y)^m=\sum_{\ell=0}^m  {\mathcal J}_{s,p,q}(m,\ell,0) \,Y^{m-\ell}X^{\ell},
$$
where  we have from \eqref{eqBinoMain} that
$$
{\mathcal J}_{s,p,q}(m,\ell,0)= p^{{\mathcal P}_{\ell}^{\ast}({\bf 0})}\sum\limits_{\lambda \in \mathcal{I}_{m-\ell,\ell}}  q^{{\mathcal Q}_{\ell|s}^{\ast}({\bf n}_{\lambda},{\bf 0})} {\mathcal A}_{\ell|s;p,q}({\bf n}_{\lambda},{\bf 0}),
$$
and where ${\bf n}_{\lambda}=(\lambda_1-\lambda_2,\lambda_2-\lambda_3,\ldots,\lambda_{\ell-1}-\lambda_{\ell},\lambda_{\ell} )$ for $\lambda \in \mathcal{I}_{m-\ell,\ell}$. Now, from the definitions it follows immediately that ${\mathcal P}_{\ell}^{\ast}({\bf 0})=0$ as well as ${\mathcal A}_{\ell|s;p,q}({\bf n}_{\lambda},{\bf 0})=1$, yielding
$$
{\mathcal J}_{s,p,q}(m,\ell,0)= \sum\limits_{\lambda \in \mathcal{I}_{m-\ell,\ell}}  q^{{\mathcal Q}_{\ell|s}^{\ast}({\bf n}_{\lambda},{\bf 0})}.
$$
The definition of ${\mathcal Q}_{r|s}^{\ast}({\bf n},{\bf k})$ shows that ${\mathcal Q}_{r|s}^{\ast}({\bf n}_{\lambda},{\bf 0})=\sum_{j=1}^{\ell} \sum_{l=1}^{j}n_l $, where $n_l=(\lambda_{\ell -l +1}-\lambda_{\ell -l +2})$ with $\lambda_{\ell +1}=0$. Thus, $\sum_{l=1}^{j}n_l =\lambda_{\ell - j +1}$. It follows that ${\mathcal Q}_{r|s}^{\ast}({\bf n}_{\lambda},{\bf 0})=\lambda_1+\cdots+\lambda_{\ell}=|\lambda|$, hence
$$
{\mathcal J}_{s,p,q}(m,\ell,0)= \sum\limits_{\lambda \in \mathcal{I}_{m-\ell,\ell}}  q^{|\lambda|}={m\choose \ell}_{q},
$$
showing the assertion.
\end{proof}

\subsection{The Weyl algebra, the $q$-Weyl algebra and the $(p,q)$-analog}\label{BinomWeyl}
\subsubsection{The Weyl algebra}\label{SecWeylBin} For the Weyl algebra, we can recover the following well-known result (see, e.g., \cite[Equation (6.28)]{TMMS2016} or \cite[Section 6.1.1]{TMMS2016} for some references) which we already observed above \eqref{Yamazaki}.
\begin{corollary}\label{WeylBinCor} Let $s=0, p=1,q=1$. Then $Z_1=I$ and \eqref{pqdefgenWeyl} reduce to the commutation relation $XY=YX+hI$ of the Weyl algebra, and one has 
$$
(X+Y)^m=\sum_{\ell=0}^m \sum_{k=0}^{\ell} h^{k} \left\{ m \atop \ell \right\}_k \,Y^{m-\ell-k}X^{\ell-k},
$$
where the {\em Weyl binomial coefficients} $\left\{ m \atop \ell \right\}_k$ are given by (see \cite[Equation (6.30)]{TMMS2016})
\begin{equation}\label{WeylBinCoEx}
\left\{ m \atop \ell \right\}_k=\frac{m!}{2^k k! (\ell -k)!(m-\ell - k)!}.
\end{equation}
\end{corollary}
\begin{proof}
In this case $Z_1=I$ and \eqref{pqdefgenWeyl} reduce to the commutation relation $XY=YX+hI$ of the Weyl algebra. From \eqref{BinomMain}, one finds
\begin{equation}\label{WeylBinom}
(X+Y)^m=\sum_{\ell=0}^m \sum_{k=0}^{\ell} h^{k} {\mathcal J}_{0,1,1}(m,\ell,k) \,Y^{m-\ell-k}X^{\ell-k}.
\end{equation}
The structure is already as expected. It remains to be shown that ${\mathcal J}_{0,1,1}(m,\ell,k)=\left\{ m \atop \ell \right\}_k$. Recall from \cite[Theorem 5.1]{AV2005} that $\left\{ m \atop \ell \right\}_k$ can be expressed as the sum of the $k$-th rook numbers over all Ferrers boards $B$ contained in $[m-\ell ] \times [\ell ]$, i.e., 
\begin{equation}\label{WBinRook}
\left\{ m \atop \ell \right\}_k=\sum_{B \subset [m-\ell ] \times [\ell ]} r_k(B).
\end{equation}
For the right-hand side, we can equivalently sum over all partitions $\lambda \in \mathcal{I}_{m-\ell,\ell}$ and consider the associated Ferrers boards $B_{\lambda}$ having $\lambda_j$ boxes in the $j$-th column, for $1\leq j \leq \ell$. Thus, it remains to show that
\begin{equation}\label{WeylBinomRel}
{\mathcal J}_{0,1,1}(m,\ell,k) =\sum_{\lambda \in \mathcal{I}_{m-\ell,\ell}} r_k(B_{\lambda}).
\end{equation}
This is exactly the contents of Corollary~\ref{BinomialRook} for $s=0$.
\end{proof}

Combining \eqref{WeylBinCoEx} and \eqref{WBinRook} yields (note that Varvak \cite[Theorem 5.1]{AV2005} gave a combinatorial proof) 
\begin{equation}\label{IDWeylVarvak}
\sum_{B \subset [m-\ell ] \times [\ell ]} r_k(B)=\frac{m!}{2^k k! (\ell -k)!(m-\ell - k)!}.
\end{equation}
 
Let us consider the concrete representation of the Weyl algebra by operators $X$ and $D$ with $(Xf)(x)=xf(x)$ and $(Df)(x)=f'(x)$, thus $X\rightsquigarrow D$ and $Y \rightsquigarrow X$. If we consider $X$ and $hD$, we have $(hD)X-X(hD)=hI$. In this concrete representation, \eqref{WeylBinom} reads as
\begin{equation}\label{WeylBinCon}
(X+hD)^n=\sum_{m=0}^n \sum_{k=0}^{m} h^{m} \left\{ n \atop m \right\}_k \,X^{n-m-k}D^{m-k}.
\end{equation}
Using $m'=n-m$ instead of $m$, and using the symmetry $\left\{ n \atop n-m \right\}_k=\left\{ n \atop m \right\}_k$, this can be found as \cite[Corollary 1]{JC2010}. Writing $\ell = m-k$, and changing the order of summation, we have
$$
(X+hD)^n=\sum_{\ell=0}^n \left\{ \sum_{m=\ell}^{n} h^{m-\ell} \left\{ n \atop m \right\}_{m-\ell} \,X^{n-2m+\ell} \right\} h^{\ell }D^{\ell}.
$$
Using $k=m-\ell$, the inner sum equals, upon using $\left\{ n \atop k+\ell \right\}_{k}={n\choose \ell} \left\{ n -\ell \atop k \right\}_{k}$, 
$$
\sum_{k=0}^{n-\ell} h^{k} \left\{ n \atop k+\ell \right\}_{k} \,X^{n-\ell-2k}={n\choose \ell}\sum_{k=0}^{n-\ell} h^{k} \left\{ n- \ell \atop k \right\}_{k} \,X^{n-\ell-2k}={n\choose \ell}H_{n-\ell}(X,h),
$$where we introduced the following variant of the Hermite polynomials \cite[Equation (1.14)]{JC2010},
\begin{equation}\label{DefHermite}
H_n(x,h)=\sum_{k=0}^{\lfloor \frac{n}{2}\rfloor } \left\{ n \atop k \right\}_k h^kx^{n-2k}=\sum_{k=0}^{\lfloor \frac{n}{2}\rfloor }\frac{n!}{2^k k! (n-2k)!}h^kx^{n-2k}.
\end{equation}
Thus, \eqref{WeylBinCon} is equivalent to the following variant of the {\em Burchnall identity} \cite[Theorem 1]{JC2010},
\begin{equation}\label{WeylBurch}
(X+hD)^n= \sum_{k=0}^{n} \left({n \atop k}\right) H_{n-k}(X,h) (hD)^{k},
\end{equation}
see \cite{BU1941} for the original version of this identity, and \cite{PRIT2009} for some discussions. Note that we can equivalently write $H_n(x,h)=\sum_{2k \leq n} {\mathcal J}_{0,1,1}(n,k,k) h^kx^{n-2k}$.

\subsubsection{The $q$-Weyl algebra}
As generalization of the situation considered in the preceding corollary, we can choose $s=0, p=1$ and $q$ arbitrary, implying that \eqref{pqdefgenWeyl} reduce to the commutation relation $XY=qYX+hI$ of the $q$-Weyl algebra. From \eqref{BinomMain}, one finds
$$
(X+Y)^m=\sum_{\ell=0}^m \sum_{k=0}^{\ell} h^{k} {\mathcal J}_{0,1,q}(m,\ell,k) \,Y^{m-\ell-k}X^{\ell-k},
$$
so that we identify ${\mathcal J}_{0,1,q}(m,\ell,k)$ with the $q$-Weyl binomial coefficient $\left\{ m \atop \ell \right\}_{k|q}$ defined by Varvak \cite[Definition 4]{AV2005} and determined explicitly by Cigler \cite[Equation (4.13)]{JC2010}. As a $q$-analog to \eqref{WBinRook} we find, by using \eqref{qBinomCoeff}, that
\begin{equation}\label{qWBinRook}
\left\{ m \atop \ell \right\}_{k|q}=\sum_{B \subset [m-\ell ] \times [\ell ]} R_{0,1;q}[B,k].
\end{equation}
Similar to the preceding section, we can consider the concrete representation of the $q$-Weyl algebra $X\rightsquigarrow hD_q$ and $Y \rightsquigarrow X$ with $(hD_q)X-qX(hD_q)=hI$ to obtain the $q$-analog to \eqref{WeylBinCon},
$$
(X+hD_q)^n=\sum_{m=0}^n \sum_{k=0}^{m} h^{m} \left\{ n \atop m \right\}_{k|q} \,X^{n-m-k}D_q^{m-k},
$$
which is exactly \cite[Equation (4.12)]{JC2010}. In this context, Cigler \cite[Equation (4.14)]{JC2010} defined a $q$-deformed version of the Hermite polynomials given in \eqref{DefHermite} by
\begin{equation}\label{qHermite}
H_n(x,h|q)=\sum_{k=0}^{\lfloor \frac{n}{2}\rfloor } \left\{ n \atop k \right\}_{k|q} h^kx^{n-2k}.
\end{equation}
Note that we can write this equivalently as $H_n(x,h|q)=\sum_{2k \leq n} {\mathcal J}_{0,1,q}(n,k,k) h^kx^{n-2k}$. Similar to the undeformed case, we can write
\begin{equation}\label{qBurch}
(X+hD_q)^n=\sum_{\ell=0}^n \left\{ \sum_{k=0}^{n-\ell} h^{k} \left\{ n \atop k+\ell  \right\}_{k|q} \,X^{n-\ell -2k}\right\} (hD_q)^{\ell}.
\end{equation}
In the undeformed case, we used the identity $\left\{ n \atop k+\ell \right\}_{k}={n\choose \ell} \left\{ n -\ell \atop k \right\}_{k}$ to express the inner bracket as Hermite polynomial $H_{n-\ell}(X,h)$. However, in the $q$-deformed case, it seems that no simple relation holds between $\left\{ n \atop k+\ell \right\}_{k|q}$ and $\left\{ n -\ell \atop k \right\}_{k|q}$. Thus, there is no  immediate $q$-analog to \eqref{WeylBurch} and Cigler's $q$-analog \cite[Theorem 4]{JC2010} involves slightly different polynomials. However, we can introduce polynomials depending on two parameters $n$ and $\ell$ (with $ 0 \leq \ell \leq n$) by
\begin{equation}\label{qHermite2}
\hat{H}_{n,\ell}(x,h|q):=\sum_{k=0}^{n-\ell} \left\{ n \atop k + \ell \right\}_{k|q} h^k x^{n-\ell-2k}.
\end{equation}
\begin{corollary}[A $q$-analog to Burchnall identity] Using the polynomials $\hat{H}_{n,\ell}(x,h|q)$ defined in \eqref{qHermite2}, identity \eqref{qBurch} can be written equivalently as
\begin{equation}\label{qBurch2}
(X+hD_q)^n=\sum_{\ell=0}^n \hat{H}_{n,\ell}(X,h|q) (hD_q)^{\ell}.
\end{equation}
\end{corollary} 
This can be regarded as a $q$-analog to Burchnall's identity \eqref{WeylBurch}. As mentioned above, for $q \neq 1$, there seems to be no simple relation between $\hat{H}_{n,\ell}(x,h|q)$ and $H_n(x,h|q)$, but for $q=1$ one has $\hat{H}_{n,\ell}(x,h|1)={n\choose \ell} H_{n-\ell}(x,h|1)={n\choose \ell} H_{n-\ell}(x,h)$, recovering \eqref{WeylBurch} from \eqref{qBurch2}. 

It is worthwhile to mention that another generalized $q$-deformation of Burchnall's identity was introduced in \cite[Equation (6.4)]{TMMS2011} as follows,
\begin{equation}\label{qHermite3}
(X+h\mathcal{E}_{s;1|q})^n=\sum_{j=0}^{n}\binom{n}{j}_q\mathcal{H}_{n;j}^{(s)}(X,h;q)h^j\mathcal{E}_{s;1|q}^{j},
\end{equation}
where $\mathcal{E}_{s;1|q}:=X^sD_{q}$ for $s\in \mathbb{N}_0$ and the polynomials $\mathcal{H}_{n;j}^{(s)}(X,h;q)$ are given in \cite[Equation (6.5)]{TMMS2011}. The operators $X$ and $\mathcal{E}_{s;1|q}$ satisfy the $q$-commutation relation (see \cite{TMMS2011} for more details) 
$$(h\mathcal{E}_{s;1|q})X-qX(h\mathcal{E}_{s;1|q})=hX^s.$$
For $s=0$, one has $\mathcal{E}_{0;1|q}=D_q$, and by comparing \eqref{qBurch2} with \eqref{qHermite3}, it follows that 
$$
\hat{H}_{n,\ell}(X,h|q)=\binom{n}{\ell}_q\mathcal{H}_{n;\ell}^{(0)}(X,h;q).
$$
On the other hand, for $q=1$ but arbitrary $s\in \mathbb{N}$, we note that $\mathcal{E}_{s}:=X^sD$ and $X$ satisfy $(h\mathcal{E}_{s})X-X(h\mathcal{E}_{s})=hX^s$ when acting on functions. Thus, one may use \eqref{ViskovS} to write
\begin{equation}\label{BurchS}
(X+h\mathcal{E}_{s})^n= \sum_{m=0}^n \binom{n}{m} h^{n-m}\, X^{(n-m)(s-1)} \mathfrak{B}^{(s-1)}_{n-m}(X^{2-s}/h) \mathcal{E}_{s}^{m},
\end{equation}
providing also a relation between the polynomials $\mathcal{H}_{n;ml}^{(s)}(X,h;1)$ and $\mathfrak{B}^{(s-1)}_{n-m}(X^{2-s}/h)$.
\subsubsection{A $(p,q)$-analog to the Weyl algebra}
Now, let us turn to the case $s=0$, but where $p$ and $q$ are not equal to 1. Then \eqref{pqdefgenWeyl} reduce to
\begin{equation}\label{pqWeyl}
XY-qYX = h Z_p, \,\, Z_pY=pYZ_p, \,\, XZ_p=pZ_pX, 
\end{equation}
which is a kind of $(p,q)$-deformation of the Weyl algebra. In fact, the relations \eqref{pqWeyl} (for $h=1$) are exactly the defining relations of the algebra $H_{p^{-1},q}$ considered by Gaddis \cite{JG2016} in an algebraic fashion. Specializing the parameter $p$ further, one obtains other algebras considered before. For example, letting $p=q$, we obtain (for $h=1$) the relations $XY-qYX = Z, \,ZY=qYZ, \,XZ=qZX$ discussed by Kirkman and Small \cite{EKLS1993}.

For this $(p,q)$-deformed Weyl algebra we have from \eqref{eqBinoMain} that
\begin{equation}\label{pqBurch0}
(X+Y)^m=\sum_{\ell=0}^m \sum_{k=0}^{\ell} h^{k} {\mathcal J}_{0,p,q}(m,\ell,k) \,Y^{m-\ell-k}Z_{p}^{k}X^{\ell-k}
\end{equation}
with ${\mathcal J}_{0,p,q}(m,\ell,k) $ given by \eqref{eqBinoMain}. Since ${\mathcal J}_{0,1,q}(m,\ell,k) =\left\{ n \atop \ell \right\}_{k|q}$ this motivates the introduction of $(p,q)$-{\em deformed Weyl binomial coefficients} by letting
$$
\left\{ m \atop \ell \right\}_{k|p;q}:={\mathcal J}_{0,p,q}(m,\ell,k),
$$
reducing for $p=1$ to the $q$-Weyl binomial coefficients, i.e., $\left\{ m \atop \ell \right\}_{k|1;q}=\left\{ m \atop \ell \right\}_{k|q}$. By \eqref{pqBinomCoeff}, one has
\begin{equation}
\left\{ m \atop \ell \right\}_{k|p;q} = \sum\limits_{\lambda \in \mathcal{I}_{m-\ell,\ell}} \mathscr{R}_{1; p,q}[B_{\lambda},k].
\end{equation}
In the same fashion, one can define in analogy to \eqref{qHermite} certain $(p,q)$-deformed Hermite polynomials,
$$
H_n(x,h|p;q):=\sum_{k=0}^{\lfloor \frac{n}{2}\rfloor } \left\{ n \atop k \right\}_{k|p;q} h^kx^{n-2k}=\sum_{k=0}^{\lfloor \frac{n}{2}\rfloor } {\mathcal J}_{0,p,q}(n,k,k) h^kx^{n-2k},
$$
where $H_n(x,h|1;q)=H_n(x,h|q)$ from above. However, as we saw already in the $q$-deformed case, the $q$-deformed Hermite polynomials do not appear in a $q$-analog to Buchnall's identity \eqref{qBurch2}. Let us point out that recently three slightly different variants of $(p,q)$-Hermite polynomials have been introduced in \cite{DAEA2018,AR2019,PNS2019}

A concrete representation of the $(p,q)$-Weyl algebra is given by $X\rightsquigarrow hD_{p,q}, Y \rightsquigarrow X$ and $Z_p\rightsquigarrow \epsilon_p$ with $(\epsilon_pf)(x)=f(px)$. Thus, \eqref{pqBurch0} becomes in this concrete representation
\begin{equation}\label{pqBurch}
(X+hD_{p,q})^n=\sum_{\ell=0}^n \sum_{k=0}^{\ell} h^{\ell} \left\{ n \atop \ell \right\}_{k|p;q} \,X^{n-\ell-k}\epsilon_{p}^{k}D_{p,q}^{\ell-k}.
\end{equation}
Doing the same manipulations as in the $q$-deformed case, this is equivalent to
$$
(X+hD_{p,q})^n=\sum_{\ell=0}^n \left\{ \sum_{k=0}^{n-\ell} h^{k} \left\{ n \atop k+\ell \right\}_{k|p;q} \,X^{n-\ell-2k}\epsilon_{p}^{k} \right\} (hD_{p,q})^{\ell}.
$$
Let us introduce in analogy to \eqref{qHermite2} the {\em operator-valued polynomials}
\begin{equation}\label{pqHermite}
\hat{H}_{n,\ell}(X,\epsilon_{p},h|p;q):=\sum_{k=0}^{n-\ell} \left\{ n \atop k + \ell \right\}_{k|p;q} h^k X^{n-\ell-2k} \epsilon_{p}^{k}=\sum_{k=0}^{n-\ell}  {\mathcal J}_{0,p,q}(n,k+\ell,k) h^k X^{n-\ell-2k} \epsilon_{p}^{k}.
\end{equation}
Then we have the following $(p,q)$-analog to \eqref{qBurch2}.
\begin{corollary}[A $(p,q)$-analog to Burchnall identity] Using the polynomials $\hat{H}_{n,\ell}(X,\epsilon_{p},h|p;q)$ defined in \eqref{pqHermite}, identity \eqref{pqBurch} can be written equivalently as
\begin{equation}
(X+hD_{p,q})^n=\sum_{\ell=0}^n \hat{H}_{n,\ell}(X,\epsilon_{p},h|p;q) (hD_{p,q})^{\ell}.
\end{equation}
\begin{remark}
It is not unusual to introduce in the $(p,q)$-deformed situation certain operator valued coefficients, in particular in the physical literature. Denoting by $S_{p,q}(n,k)$ the $(p,q)$-deformed Stirling numbers of the second kind, one has in the present context the normal ordering result (see \cite{TLM1,LO2020,LO2021})
$$
(XD_{p,q})^n=\sum_{k=0}^n S_{p,q}(n,k)X^k \varepsilon_p^{n-k}D_{p,q}^k.
$$ 
Using $X^k \varepsilon_p^{n-k}=p^{-k(n-k)}\varepsilon_p^{n-k}X^k $ and introducing {\em operator valued Stirling numbers} 
$$\hat{S}_{p,q}(n,k; \varepsilon_p)=p^{-k(n-k)}S_{p,q}(n,k)\varepsilon_p^{n-k},$$
we can write this in standard form $(XD_{p,q})^n=\sum_{k=0}^n \hat{S}_{p,q}(n,k; \varepsilon_p) X^k D_{p,q}^k$. This point of view was introduced by Katriel and Kibler \cite{KK1992}, see also the recent papers \cite{HOT2026,OT2024} and the literature given therein.
\end{remark}
\end{corollary}
Let us note that by choosing $p=q^2$ and $h=1$ in this case, the commutation relations \eqref{pqWeyl} become
\begin{equation}\label{Blumen}
XY-qYX = Z, \,\, ZY=q^2YZ, \,\, XZ=q^2ZX.
\end{equation}
For variables satisfying these relations, Blumen \cite[Lemma 2.1]{SB2006} derived the following binomial formula,
\begin{equation}\label{BinBlumen}
(X+Y)^n=\sum_{\alpha,\beta,\gamma \in \mathbb{Z}_+ \atop \alpha+2\beta+\gamma=n}\frac{[n]_q!}{[\alpha]_q![\gamma]_q! [2]_q[4]_q\cdots [2\beta]_q}Y^{\alpha}Z^{\beta}X^{\gamma}. 
\end{equation}
Note that this reduces for $Z=0$ to \eqref{BFPS}. For $p=q^2$ and $h=1$ we have from \eqref{pqBurch0} (with $Z \equiv Z_{q^2}$)
\begin{equation}\label{Binpqqh1}
(X+Y)^n=\sum_{\ell=0}^n \sum_{k=0}^{\ell} {\mathcal J}_{0,q^2,q}(n,\ell,k) \,Y^{n-\ell-k}Z^{k}X^{\ell-k}.
\end{equation}
Observe that the exponents in Equation \eqref{Binpqqh1} satisfy $(n-\ell-k)+2k+(\ell -k)=n$, i.e., the same condition as in \eqref{BinBlumen}. Hence, by assuming that ${\mathcal J}_{0,p,q}(n,\ell,k)=0$ for $n-l-k<0$, we can write \eqref{Binpqqh1} as
\begin{equation}\label{Binpqqh2}
(X+Y)^n=\sum_{k, (\ell-k)\in \mathbb{Z}_+ \atop (n-\ell-k) \in\mathbb{Z}_+}{\mathcal J}_{0,q^2,q}(n,\ell,k)Y^{n-\ell-k}Z^kX^{\ell-k}.
\end{equation}
Comparing \eqref{BinBlumen} with \eqref{Binpqqh2}, it follows that 
\begin{equation}\label{NOCBlumen}
{\mathcal J}_{0,q^2,q}(n,\ell,k)=\frac{[n]_q!}{[n-\ell-k]_q![\ell-k]_q![2]_q[4]_q\cdots[2k]_q}.
\end{equation}
\begin{remark}
Recalling the connection between ${\mathcal J}_{0,q^2,q}(n,\ell,k)$ and the $(p,q)$-deformed rook numbers $\mathscr{R}_{1; q^2,q}[B_{\lambda},k]$, see \eqref{pqBinomCoeff}, it would be interesting to give a combinatorial proof of \eqref{NOCBlumen}.
\end{remark}
\subsection{The shift algebra, the $q$-shift algebra and the $(p,q)$-analog}\label{BinomShift}
For $s=1$, one has from \eqref{BinomMain}
\begin{equation}\label{pqShiftBin}
(X+Y)^m=\sum_{\ell=0}^m \sum_{k=0}^{\ell} h^{k} {\mathcal J}_{1,p,q}(m,\ell,k) \,Y^{m-\ell}Z_{p}^{k}X^{\ell-k},
\end{equation}
where ${\mathcal J}_{1,p,q}(m,0,k)=1$ and, for $\ell \neq 0$, 
\begin{equation}\label{pqShiftBinC}
{\mathcal J}_{1,p,q}(m,\ell,k)=\sum\limits_{{\bf k}\in \{0,1\}^{\ell} \atop |{\bf k}|=k}  p^{{\mathcal P}_{\ell}^{\ast}({\bf k})}\sum\limits_{\lambda \in \mathcal{I}_{m-\ell,\ell}}  q^{{\mathcal Q}_{\ell|1}^{\ast}({\bf n}_{\lambda},{\bf k})} {\mathcal A}_{\ell|1;p,q}({\bf n}_{\lambda},{\bf k}),
\end{equation}
with 
$$
{\mathcal P}_{\ell}^{\ast}({\bf k})=\sum_{j=1}^{\ell-1}(1-k_{j+1})\sum_{l=1}^{j}k_l, \,\,\, {\mathcal Q}_{\ell|1}^{\ast}({\bf n}_{\lambda},{\bf k})=|\lambda|-\sum_{j=1}^{\ell}k_j\lambda_{\ell-j+1}, \,\,\, {\mathcal A}_{\ell|1;p,q}({\bf n}_{\lambda},{\bf k})=\prod_{j=1}^{\ell}\left[\lambda_{\ell - j+1} \right]_{p, q}^{k_j}.
$$
\subsubsection{The shift algebra} Let us specialize to $p=q=1$. From Corollary~\ref{BinomialRook}, we obtain for $s=1$ that 
$$
{\mathcal J}_{1,1,1}(m,\ell,k)=\sum\limits_{\lambda \in \mathcal{I}_{m-\ell,\ell}}r_k^{(1)}(B_{\lambda})=\sum\limits_{\lambda \in \mathcal{I}_{m-\ell,\ell}} f_k(B_{\lambda}),
$$
where we used that the 1-rook numbers are equal to the file numbers, $r_k^{(1)}(B_{\lambda})=f_k(B_{\lambda})$, see \eqref{RookRed}. Similar to above, we can write this equivalently as a sum over all boards $B$ contained in $[m-\ell]\times [\ell]$. If we define in analogy to \eqref{WBinRook} the {\em shift binomial coefficients} by
\begin{equation}\label{WBinFile}
\left \langle m \atop \ell \right\rangle_k :=\sum_{B \subset [m-\ell ] \times [\ell ]} f_k(B),
\end{equation}
then we find in analogy to \eqref{WeylBinomRel} that
\begin{equation}\label{ShiftBinomRel}
{\mathcal J}_{0,1,1}(m,\ell,k) =\sum_{\lambda \in \mathcal{I}_{m-\ell,\ell}} f_k(B_{\lambda})=\left \langle m \atop \ell \right\rangle_k .
\end{equation}
Thus, we have shown the following analog to Corollary~\ref{WeylBinCor} holding in the Weyl algebra.
\begin{corollary}\label{ShiftBinCor} Let $s=1, p=1,q=1$. Then $Z_1=I$ and \eqref{pqdefgenWeyl} reduce to the commutation relation $XY=YX+hY$ of the shift algebra, and one has 
$$
(X+Y)^m=\sum_{\ell=0}^m \sum_{k=0}^{\ell} h^{k} \left \langle m \atop \ell \right\rangle_k \,Y^{m-\ell} X^{\ell-k},
$$
where $\left \langle m \atop \ell \right\rangle_k$ are the shift binomial coefficients defined in \eqref{WBinFile}.
\end{corollary}
Switching the order of summation and relabelling the indices, the identity of Corollary~\ref{ShiftBinCor} equals 
\begin{equation}\label{BinomShiftRe}
(X+Y)^m=\sum_{\ell =0}^m \sum_{k=0}^{\ell} h^{\ell-k}  \left \langle m \atop m-k\right\rangle_{\ell-k} \,Y^{k} X^{m-\ell}.
\end{equation}
Writing in Viskov's result \eqref{VisShift} the $k$-the Bell polynomial $B_k$ explicitly, one obtains
\begin{equation}\label{Viskov}
(X+Y)^m=\sum_{\ell=0}^m \sum_{k=0}^{\ell} \binom{m}{\ell} h^{\ell -k}  S(\ell,k) \, Y^k X^{m-\ell}.
\end{equation} 
Comparing \eqref{BinomShiftRe} and \eqref{Viskov}, we obtain $\left \langle m \atop m-k\right\rangle_{\ell-k} =\binom{m}{\ell} S(\ell,k) $. Some index manipulations yield the following result.
\begin{proposition}\label{ShiftBinExP} The shift binomial coefficients defined in \eqref{WBinFile} are given explicitly by
\begin{equation}\label{ShiftBinEx}
\left \langle m \atop \ell \right\rangle_{k} =\binom{m}{m-\ell+k} S(m-\ell+k,m-\ell). 
\end{equation}
\end{proposition}
Combining \eqref{WBinFile} and \eqref{ShiftBinEx}, we obtain the following analog to \eqref{IDWeylVarvak},
\begin{equation}\label{IDShiftVarvak}
\sum_{B \subset [m-\ell ] \times [\ell ]} f_k(B)=\binom{m}{m-\ell+k} S(m-\ell+k,m-\ell).
\end{equation}
\begin{remark} Varvak's proof \cite[Theorem 5.1]{AV2005} of \eqref{IDWeylVarvak} was in a combinatorial fashion. The above derivation of \eqref{IDShiftVarvak} was slightly indirect and algebraically. It would be nice to have a combinatorial proof analogous to the one of \eqref{IDWeylVarvak}.
\end{remark}

Similar to the case of the Weyl algebra, where an operational interpretation of the binomial formula is given by the Burchnall identity \eqref{WeylBurch}, we can use here the operators $X$ and $hXD$ (the {\em Euler operator}). They satisfy $(hXD)X-X(hXD)=hX$, i.e., they give a concrete representation of the commutation relations of the shift algebra (see \cite[Section 8.1.2]{TMMS2016} and \cite[Section 8.4.3]{TMMS2016} for some references). Using Viskov's identity \eqref{VisShift}, we obtain (observe that this is the case $s=1$ of \eqref{BurchS})
$$
(X+hXD)^m=\sum_{\ell=0}^m \binom{m}{\ell} h^{m-\ell} \, B_{m-\ell}(X/h)(hXD)^{\ell}. 
$$
Expanding the Bell polynomials and using the normal ordering result $(XD)^n=\sum_{k}S(n,k)X^kD^k$, this yields the explicit expression (derived in an equivalent form by Mikhailov \cite{VVM1985})
\begin{equation}
(X+hXD)^m=\sum_{k=0}^{m}\sum_{\ell=k}^m \sum_{r=0}^{m-\ell} h^{m-r} \binom{m}{\ell}  S(m-\ell,r) S(\ell,k)X^{k+r}D^k. 
\end{equation}
Very similar expressions have been determined by Mikhailov \cite{VVM1983} for the expansion of $(D+XD)^m$, see the discussion in \cite[Section 6.1.1]{TMMS2016}. As early as 1887, d'Ocagne \cite{MO1887} considered an expansion of $(X+DX)^m$.

\subsubsection{The $q$-shift algebra} Let us specialize to $p=1$. Then \eqref{pqShiftBinC} reduces to
\begin{equation}\label{qShiftBin}
(X+Y)^m=\sum_{\ell=0}^m \sum_{k=0}^{\ell} h^{k} {\mathcal J}_{1,1,q}(m,\ell,k) \,Y^{m-\ell}X^{\ell-k},
\end{equation}
where ${\mathcal J}_{1,1,q}(m,\ell,k)$ is given by \eqref{qBinomCoeff} (for $s=1$). 
 
Let us denote by $f_k(B_{\lambda}|q):=R_{1,1;q}[B_{\lambda},k]$ the $q$-deformed $k$-th file number of the board $B_{\lambda}$ associated to $\lambda \in \mathcal{I}_{m-\ell,\ell}$. We can then define a $q$-analog to \eqref{WBinFile} -- the {\em $q$-shift binomial coefficients} - by
\begin{equation}\label{qWBinFile}
\left \langle m \atop \ell \right\rangle_{k|q} :=\sum_{B \subset [m-\ell ] \times [\ell ]} f_k(B|q).
\end{equation}
Clearly, ${\mathcal J}_{1,1,q}(m,\ell,k)=\left \langle m \atop \ell \right\rangle_{k|q}$. 
\begin{corollary}\label{qShiftBinCor} Let $s=1, p=1$. Then $Z_1=I$ and \eqref{pqdefgenWeyl} reduce to the commutation relation $XY=qYX+hY$ of the $q$-shift algebra, and one has 
$$
(X+Y)^m=\sum_{\ell=0}^m \sum_{k=0}^{\ell} h^{k} \left \langle m \atop \ell \right\rangle_{k|q} \,Y^{m-\ell} X^{\ell-k},
$$
where $\left \langle m \atop \ell \right\rangle_{k|q}$ are the $q$-shift binomial coefficients defined in \eqref{qWBinFile}.
\end{corollary}

\subsubsection{A $(p,q)$ analog to the shift algebra} In the general case, \eqref{pqShiftBin} holds true and we have to consider the coefficients ${\mathcal J}_{1,p,q}(m,\ell,k)$ from \eqref{pqShiftBinC}, given explicitly by
\begin{equation}
{\mathcal J}_{1,p,q}(m,\ell,k)=\sum\limits_{\lambda \in \mathcal{I}_{m-\ell,\ell}} \sum\limits_{{\bf k}\in \{0,1\}^{\ell} \atop |{\bf k}|=k}  p^{\sum_{j=1}^{\ell-1}(1-k_{j+1})\sum_{l=1}^{j}k_l} q^{|\lambda|-\sum_{j=1}^{\ell}k_j\lambda_{\ell-j+1}} \prod_{j=1}^{\ell}\left[\lambda_{\ell - j+1} \right]_{p, q}^{k_j}.
\end{equation}
Instead of using this explicit form, we recall the identification given by \eqref{pqBinomCoeff}. Thus, if we define by 
$$
f_k(B_{\lambda}|p;q):= \mathscr{R}_{1,1; p,q}[B_{\lambda},k].
$$
a $(p,q)$-analog to the $k$-th file number of the board $B_{\lambda}$ associated to $\lambda \in \mathcal{I}_{m-\ell,\ell}$, then we can define a $(p,q)$-analog to \eqref{WBinFile} -- the {\em $(p,q)$-shift binomial coefficients} -- by
\begin{equation}\label{pqWBinFile}
\left \langle m \atop \ell \right\rangle_{k|p;q} :=\sum_{B \subset [m-\ell ] \times [\ell ]} f_k(B|p;q)=\sum_{B \subset [m-\ell ] \times [\ell ]} \mathscr{R}_{1,1; p,q}[B,k].
\end{equation}
Clearly, ${\mathcal J}_{1,p,q}(m,\ell,k)=\left \langle m \atop \ell \right\rangle_{k|p;q}$. 
\begin{corollary}\label{pqShiftBinCor} Let $s=1$. Then \eqref{pqdefgenWeyl} reduce to the commutation relations $XY=qYX+hYZ_p, \, Z_pY=pYZ_p, \, XZ_p=pZ_pX, \,\, Z_1 =I,$ of the $(p,q)$-analog to the shift algebra, and one has 
$$
(X+Y)^m=\sum_{\ell=0}^m \sum_{k=0}^{\ell} h^{k} \left \langle m \atop \ell \right\rangle_{k|p;q} \,Y^{m-\ell} Z_p^{k} X^{\ell-k},
$$
where $\left \langle m \atop \ell \right\rangle_{k|p;q}$ are the $(p,q)$-shift binomial coefficients defined in \eqref{pqWBinFile}.
\end{corollary}

\subsection{The Jordan plane, the $q$-Jordan plane and the $(p,q)$-analog}\label{SectBinomJordan}
Let us consider $s=2$. From \eqref{BinomMain}, we obtain 
\begin{equation}\label{BinomJordan}
(X+Y)^m=\sum_{\ell=0}^m \sum_{k=0}^{\ell} h^{k} {\mathcal J}_{2,p,q}(m,\ell,k) \,Y^{m-\ell+k}Z_{p}^{k}X^{\ell-k},
\end{equation}
where (see \eqref{eqBinoMain})
$$
{\mathcal J}_{2,p,q}(m,\ell,k)=\sum\limits_{\lambda \in \mathcal{I}_{m-\ell,\ell}} \sum\limits_{{\bf k}\in \{0,1\}^{\ell} \atop |{\bf k}|=k}  p^{{\mathcal P}_{\ell}^{\ast}({\bf k})} q^{{\mathcal Q}_{\ell|2}^{\ast}({\bf n}_{\lambda},{\bf k})} {\mathcal A}_{\ell|2;p,q}({\bf n}_{\lambda},{\bf k}),
$$
with ${\bf n}_{\lambda}=(\lambda_1-\lambda_2,\lambda_2-\lambda_3,\ldots,\lambda_{\ell-1}-\lambda_{\ell},\lambda_{\ell} )$ for $\lambda \in \mathcal{I}_{m-\ell,\ell}$.

\subsubsection{The Jordan plane} Let us consider $p=q=1$. If we define, in analogy to above, the {\em Jordan binomial coefficients} by 
\begin{equation}\label{JordanBinFile}
\left\langle \left\langle {m \atop \ell } \right\rangle \right\rangle_k := \sum\limits_{\lambda \in \mathcal{I}_{m-\ell,\ell}}r_k^{(2)}(B_{\lambda}),
\end{equation}
then $\langle\langle {m \atop \ell } \rangle\rangle_k ={\mathcal J}_{2,1,1}(m,\ell,k)$ by Corollary~\ref{BinomialRook} for $s=2$. Thus, we have the following result.
\begin{corollary}\label{JordanBinCor} Let $s=2, p=1,q=1$. Then $Z_1=I$ and \eqref{pqdefgenWeyl} reduce to the commutation relation $XY=YX+hY^2$ of the Jordan plane, and one has 
$$
(X+Y)^m=\sum_{\ell=0}^m \sum_{k=0}^{\ell} h^{k} \left \langle  \left\langle m \atop \ell \right\rangle \right\rangle_k \,Y^{m-\ell+k} X^{\ell-k},
$$
where $\left \langle \left \langle m \atop \ell \right\rangle\right\rangle_k$ are the Jordan binomial coefficients defined in \eqref{JordanBinFile}.
\end{corollary}
By manipulating the summation indices, we can write \eqref{BinomJordan} equivalently in the following form,
\begin{equation}\label{BinomShiftGen}
(X+Y)^m=\sum_{k=0}^{m} \left\{\sum_{j=0}^{k} h^{j} {\mathcal J}_{2,1,1}(m,m-k+j,j) \right\}\,Y^{k}X^{m-k}. 
\end{equation}
We can now compare this to the binomial formula derived by Benaoum \cite{HBB1998} and Viskov \cite{OVV1998} (see also \cite[Section 7.3]{TMMS2016} for more references in this context). Defining the {\em $h$-binomial coefficients} by
$$
\binom{n}{k}_h:=\binom{n}{k}\prod_{i=0}^{k-1}(1+hi),
$$
Benaoum \cite{HBB1998} wrote the binomial formula in this context in the beautiful form
\begin{equation}\label{BinomBen}
(X+Y)^m=\sum_{k=0}^{m} \binom{m}{k}_hY^kX^{m-k},
\end{equation}
see \eqref{ViskovJordan}. Comparing \eqref{BinomShiftGen} and \eqref{BinomBen}, this shows 
\begin{equation}\label{JordanComp}
\binom{m}{k}_h=\sum_{j=0}^{k} {\mathcal J}_{2,1,1}(m,m-k+j,j) h^{j}.
\end{equation}
From $x^{\overline{k}}=\sum_{j=0}^k |s(k,j)| x^j$ we obtain the expansion in $h$,
$$
\prod_{i=0}^{k-1}(1+hi)=\sum_{j=0}^{k} |s(k,k-j)| h^j .
$$
Inserting this into \eqref{JordanComp} and comparing coefficients of $h^j$ gives ${\mathcal J}_{2,1,1}(m,m-k+j,j)=\binom{m}{k} |s(k,k-j)|$. Recalling the cycle numbers $c(n,k) \equiv |s(n,k)|$, some index manipulations yield the following analog to Proposition~\ref{ShiftBinExP}.
\begin{proposition}\label{JordanBinExP} The Jordan binomial coefficients defined in \eqref{JordanBinFile} are given explicitly by
\begin{equation}\label{JordanBinEx}
\left\langle \left\langle {m \atop \ell } \right\rangle \right\rangle_k =\binom{m}{m-\ell+k} c(m-\ell+k,m-\ell). 
\end{equation}
\end{proposition}
Combining \eqref{JordanBinFile} and \eqref{JordanBinEx}, we obtain the following analog to \eqref{IDWeylVarvak} and \eqref{IDShiftVarvak},
\begin{equation}\label{IDJordanVarvak}
\sum_{B \subset [m-\ell ] \times [\ell ]} r_k^{(2)}(B)=\binom{m}{m-\ell+k} c(m-\ell+k,m-\ell).
\end{equation}
Note that \eqref{JordanComp} gives an explicit connection between Benaoum's $h$-binomial coefficients and the Jordan binomial coefficients defined here,
$$
\binom{m}{k}_h=\sum_{j=0}^{k}\left\langle \left\langle {m \atop m-k+j } \right\rangle \right\rangle_j h^{j}.
$$

\subsubsection{The $q$-Jordan plane} Let us consider $p=1$. If we define, in analogy to above, the {\em $q$-Jordan binomial coefficients} by 
\begin{equation}\label{qJordanBinFile}
\left\langle \left\langle {m \atop \ell } \right\rangle \right\rangle_{k|q} := \sum\limits_{\lambda \in \mathcal{I}_{m-\ell,\ell}}R_{2, 1; q}[B_{\lambda}, k],
\end{equation}
where the $q$-deformed $2$-rook numbers $R_{2, 1; q}[B_{\lambda}, k]$ are defined in \eqref{RookCeleste}, then $\langle\langle {m \atop \ell } \rangle\rangle_{k|q} ={\mathcal J}_{2,1,q}(m,\ell,k)$ by Corollary~\ref{qBinomialRook} for $s=2$. Thus, we have the following result.
\begin{corollary}\label{qJordanBinCor} Let $s=2, p=1$. Then $Z_1=I$ and \eqref{pqdefgenWeyl} reduce to the commutation relation $XY=qYX+hY^2$ of the $q$-Jordan plane, and one has 
$$
(X+Y)^m=\sum_{\ell=0}^m \sum_{k=0}^{\ell} h^{k} \left \langle  \left\langle m \atop \ell \right\rangle \right\rangle_{k|q} \,Y^{m-\ell+k} X^{\ell-k},
$$
where $\left \langle \left \langle m \atop \ell \right\rangle\right\rangle_{k|q}$ are the $q$-Jordan binomial coefficients defined in \eqref{qJordanBinFile}.
\end{corollary}
As in the undeformed case, we can write \eqref{BinomJordan} equivalently in the following form,
\begin{equation}\label{qBinomShiftGen}
(X+Y)^m=\sum_{k=0}^{m} \left\{\sum_{j=0}^{k} h^{j} {\mathcal J}_{2,1,q}(m,m-k+j,j) \right\}\,Y^{k}X^{m-k}.
\end{equation}
Benaoum \cite{HBB1999} wrote the binomial formula in this context in the form
\begin{equation}\label{qBinomBen}
(X+Y)^m=\sum_{k=0}^{m} \binom{m}{k}_{h|q}Y^kX^{m-k},
\end{equation}
where the {\em $(q,h)$-deformed binomial coefficient} is defined by
\begin{equation}\label{qhBinCoeff}
\binom{n}{k}_{h|q}:=\binom{n}{k}_q \prod_{i=0}^{k-1}(1+h[i]_q).
\end{equation}
Comparing \eqref{qBinomShiftGen} and \eqref{qBinomBen}, this shows the $q$-analog to \eqref{JordanComp},
\begin{equation}\label{qJordanComp}
\binom{m}{k}_{h|q}=\sum_{j=0}^{k} {\mathcal J}_{2,1,q}(m,m-k+j,j) h^{j}. 
\end{equation}
Proceeding as in the undeformed case, we have for the unsigned $q$-deformed Stirling numbers of the first kind (or equivalently, the {\em $q$-deformed cycle numbers}), $c_q(n,k) $, that
$$
x(x+[1]_q)(x+[2]_q)\cdots (x+[k-1]_q)=\sum_{j=0}^k c_q(k,j)x^j,
$$
implying an expansion
\begin{equation}\label{qcycleexp}
\prod_{i=0}^{k-1}(1+h[i]_q)=\sum_{j=0}^{k} c_q(k,k-j) h^j. 
\end{equation}
Recall that in the undeformed case the cycle numbers were used in Proposition~\ref{JordanBinExP}. They can be written in terms of the elementary symmetric functions as $c(n,k)=e_{n-k}(1,2,\ldots,n-1)$, and one has the $q$-analog 
\begin{equation}\label{qcyclesym}
c_q(n,k)=e_{n-k}([1]_q,[2]_q,\ldots,[n-1]_q).
\end{equation} 
Thus, we have the following $q$-analog to Proposition~\ref{JordanBinExP}.
\begin{proposition} The $q$-Jordan binomial coefficients defined in \eqref{qJordanBinFile} are given explicitly by
\begin{equation}\label{qJordanBinEx}
\left\langle \left\langle {m \atop \ell } \right\rangle \right\rangle_{k|q} =\binom{m}{m-\ell+k}_q c_q(m-\ell+k,m-\ell). 
\end{equation}
\end{proposition}
Combining \eqref{qJordanBinFile} and \eqref{qJordanBinEx}, we obtain the following $q$-analog to \eqref{IDJordanVarvak},
\begin{equation}\label{qIDJordanVarvak}
\sum_{B \subset [m-\ell ] \times [\ell ]} R_{2, 1; q}[B, k]=\binom{m}{m-\ell+k}_q c_q(m-\ell+k,m-\ell).
\end{equation}
\begin{remark}
Further properties of the $(q,h)$-binomial coefficients can be found in \cite{ZZJW2000}. Let us point out that they reduce for $h=q-1$ due to $\prod_{i=0}^{k-1}(1+(q-1)[i]_q)=q^{\binom{k}{2}}$ nicely to $\binom{n}{k}_{q-1|q}=\binom{n}{k}_q q^{\binom{k}{2}}$, as was observed by Rosengren \cite{HR2000}, see also \cite{AHMM2002}. In that case, \eqref{qBinomBen} reduces to 
$$
(X+Y)^m=\sum_{k=0}^{m} \binom{m}{k}_{q}q^{\binom{k}{2}}\, Y^kX^{m-k}.
$$
\end{remark}

\subsubsection{A $(p,q)$-analog to the Jordan plane}\label{SectpqJordan2} Let us consider $p\neq 1\neq q$ and define in analogy to above the {\em $(p, q)$-Jordan binomial coefficients} by 
\begin{equation}\label{pqJordanBinFile}
\left\langle \left\langle {m \atop \ell } \right\rangle \right\rangle_{k|p; q} := \sum\limits_{\lambda \in \mathcal{I}_{m-\ell,\ell}}\mathscr{R}_{2, 1;p, q}[B_{\lambda}, k], \text{ and } \left\langle \left\langle {m \atop \ell } \right\rangle \right\rangle_{k|p; q}=0 \text{ for } m-\ell-k<0,
\end{equation}
where the $(p, q)$-deformed 2-rook numbers $\mathscr{R}_{2, 1;p, q}[B_{\lambda}, k]$ are given by \eqref{PQRookWeS}. Then, according to Corollary~\ref{pqBinomialRook}, for $s=2$, one has
\begin{equation}\label{pqJordanDefin}
\left\langle \left\langle {m \atop \ell } \right\rangle \right\rangle_{k|p; q}={\mathcal J}_{2,p,q}(m,\ell,k).
\end{equation}
\begin{corollary}\label{pqJordanBinCor} Let $s=2, p\neq 1\neq q$. Then \eqref{pqdefgenWeyl} reduce to the commutation relations $XY=qYX+hY^2Z_p, \,  XZ_p=pZ_pX, \, Z_pY=pYZ_p$ of the $(p, q)$-Jordan plane, and one has 
$$
(X+Y)^m=\sum_{\ell=0}^m \sum_{k=0}^{\ell} h^{k} \left \langle  \left\langle m \atop \ell \right\rangle \right\rangle_{k|p; q} \,Y^{m-\ell+k} Z_p^kX^{\ell-k},
$$
where $\left \langle \left \langle m \atop \ell \right\rangle\right\rangle_{k|p; q}$ are the $(p, q)$-Jordan binomial coefficients defined in \eqref{pqJordanBinFile}.
\end{corollary}
As in the above cases, we can write \eqref{BinomJordan} equivalently in the following from
$$
(X+Y)^m=\sum_{k=0}^{m} \left\{\sum_{j=0}^{k} h^{j} {\mathcal J}_{2,p,q}(m,m-k+j,j) Y^{k}Z_{p}^{j}\right\}\,X^{m-k}.
$$
Thus, using $Y^kZ^j=p^{-kj}Z^jY^k$, we have found that
\begin{equation}\label{pqBinomJordanGen}
(X+Y)^m=\sum_{k=0}^{m} \left\{\sum_{j=0}^{k} h^{j} p^{-kj} {\mathcal J}_{2,p,q}(m,m-k+j,j) Z_{p}^{j}\right\}\,Y^{k}X^{m-k}.
\end{equation}
In view of \eqref{qJordanComp}, it is natural to introduce the $(p,q,h)$-{\it deformed binomial coefficients} by
\begin{equation}\label{pqJordanComp}
\binom{m}{k}_{h|p;q}(Z_p) :=\sum_{j=0}^{k} {\mathcal J}_{2,p,q}(m,m-k+j,j) p^{-kj}  h^{j} Z_{p}^{j},
\end{equation}
with $\binom{m}{k}_{h|1;q}(Z_1) =\binom{m}{k}_{h|q}$. Thus, we can write \eqref{pqBinomJordanGen} in a form similar to Benaoum's formula \eqref{qBinomBen},
\begin{equation}\label{pqBinomBen}
(X+Y)^m=\sum_{k=0}^{m} \binom{m}{k}_{h|p;q}(Z_p)\,\,  Y^{k} X^{m-k}.
\end{equation}
Note that we denoted the dependence on $Z_p$ explicitly. Let us write
\begin{equation}\label{pqBinomEx}
\binom{m}{k}_{h|p;q}(Z_p) =\binom{m}{k}_{p,q}\sum_{j=0}^{k} \binom{m}{k}_{p,q}^{-1}{\mathcal J}_{2,p,q}(m,m-k+j,j) p^{-kj}  h^{j} Z_{p}^{j}.
\end{equation}
In analogy to the $q$-deformed situation \eqref{qcyclesym} we define the $(p,q)$-deformed cycle numbers by
\begin{equation}\label{pqcyclestan}
c_{p,q}(n,k):=e_{n-k}([1]_{p,q},[2]_{p,q},\ldots,[n-1]_{p,q},
\end{equation}
see \cite{dML1993}. Then we obtain the following analog to \eqref{qcycleexp},
\begin{equation}\label{pqcycleexp}
\prod_{i=0}^{k-1}(1+h[i]_{p,q})=\sum_{j=0}^{k} c_{p,q}(k,k-j) h^j. 
\end{equation}
Replacing $h$ in \eqref{pqcycleexp} by $hp^{-k}Z_p$, we obtain
$$
\prod_{i=0}^{k-1}(1+hp^{-k}[i]_{p,q}Z_p)=\sum_{j=0}^{k} c_{p,q}(k,k-j) p^{-kj}h^j Z_p^j.
$$
\begin{conjecture} The $(p,q,h)$-deformed binomial coefficients can be written as 
$$
\binom{m}{k}_{h|p;q}(Z_p) =\binom{m}{k}_{p,q} \prod_{i=0}^{k-1}(1+hp^{-k}[i]_{p,q}Z_p).
$$
\end{conjecture}
\begin{remark} 
Comparing this with \eqref{pqBinomEx} would imply
${\mathcal J}_{2,p,q}(m,m-k+j,j)=\binom{m}{k}_{p,q}c_{p,q}(k,k-j)$. Using \eqref{pqJordanDefin} and some index manipulations as above, this identity is equivalent to
$$
\left\langle \left\langle {m \atop \ell } \right\rangle \right\rangle_{k|p; q}=\binom{m}{m-\ell +k}_{p,q}c_{p,q}(m-\ell +k,m-\ell),
$$
being the natural $(p,q)$-analog to \eqref{qJordanBinEx}, implying also the natural $(p,q)$-analog to \eqref{qIDJordanVarvak}.
\end{remark}

\subsection{A few remarks on the case $s=3$}\label{Sect3S} Let us consider $s=3$ and $p=q=1$, i.e., we consider the commutation relation $XY-YX=hY^3$. It seems that one does not have a nice interpretation of the variables as in the cases $s=0,1,2$ (but note that in an operational realization one has $DX^{-\frac{1}{2}}-X^{-\frac{1}{2}}D=-\frac{1}{2}(X^{-\frac{1}{2}})^3$, see Remark~\ref{RemFrac} and \cite{MS2024}). However, if we define in analogy to above, 
\begin{equation}\label{3BinFile}
\left\{ \left\{ {m \atop \ell } \right\} \right\}_k := \sum\limits_{\lambda \in \mathcal{I}_{m-\ell,\ell}}r_k^{(3)}(B_{\lambda}),
\end{equation}
then $\left\{ \left\{ {m \atop \ell } \right\} \right\}_k ={\mathcal J}_{3,1,1}(m,\ell,k)$ by Corollary~\ref{BinomialRook} for $s=3$. Thus, we have the following result.
\begin{corollary}\label{3BinCor} Let $s=3, p=1,q=1$. Then $Z_1=I$ and \eqref{pqdefgenWeyl} reduce to the commutation relation $XY=YX+hY^3$ and one has 
$$
(X+Y)^m=\sum_{\ell=0}^m \sum_{k=0}^{\ell} h^{k} \left\{ \left\{ {m \atop \ell } \right\} \right\}_k \,Y^{m-\ell+2k} X^{\ell-k}.
$$
\end{corollary}
Switching the order of summation and relabelling the indices, we can write equivalently
\begin{equation}\label{3Expansion}
(X+Y)^m=\sum_{k=0}^{m} \sum_{\ell=0}^{m-k}  h^{m-k-\ell} \left\{ \left\{ {m \atop m- \ell } \right\} \right\}_{m-\ell-k} \,Y^{2m-2k-\ell} X^{k}.
\end{equation}
Similar to the preceding sections, we compare this binomial formula (for $h=-1$) to the explicit formula given in Proposition~\ref{Bessel}. Recall \cite{SYZQ2011} that one introduces the signless Bessel numbers of the first kind $a(n,k)$ as coefficient of $x^{n-k}$ in $y_{n-1}$, and the Bessel numbers of the first kind $b(n,k)$ by $b(n,k)=(-1)^{n-k}a(n,k)$. Thus, writing 
\begin{equation}\label{BessPol}
y_{n}(x)=\sum_{k=0}^{n} \frac{(2n-k)!}{2^{n-k} (n-k)! k!}x^{n-k},
\end{equation}
we find the well-known result \cite{SYZQ2011} 
\begin{equation}\label{BesselExp}
a(n,k)=\frac{(2n-k-1)!}{2^{n-k} (n-k)! (k-1)!}.
\end{equation}
Using in \eqref{bf2823b} that
$$
y_{n-m-1}(-Y)=\sum_{k=1}^{n-m}a(n-m,k)(-Y)^{n-m-k},
$$
we obtain, by adapting also the indices for better comparison with \eqref{3Expansion}, from \eqref{bf2823b} finally
\begin{equation}\label{3ExpansionVis}
(X+Y)^m= \sum_{k=0}^m \sum_{\ell=0}^{m-k}\binom{m}{k} a(m-k,\ell)(-1)^{m-k-\ell} Y^{2m-2k-\ell} X^k. 
\end{equation}
Thus, comparing \eqref{3Expansion} (for $h=-1$) with \eqref{3ExpansionVis}, one finds
$$
\left\{ \left\{ {m \atop m- \ell } \right\} \right\}_{m-\ell-k}=\binom{m}{k} a(m-k,\ell).
$$
Some index manipulations yield the following analog to Proposition~\ref{ShiftBinExP}.
\begin{proposition}\label{3BinExP} The coefficients defined in \eqref{3BinFile} are given explicitly by
\begin{equation}\label{3BinEx}
\left\{ \left\{ {m \atop \ell } \right\} \right\}_k = \binom{m}{m-\ell+k} a(m-\ell+k,m-\ell), 
\end{equation}
where $a(n,k)$ denotes the signless Bessel numbers of the first kind.
\end{proposition}
Combining \eqref{3BinFile} and \eqref{3BinEx}, we obtain the following analog to \eqref{IDWeylVarvak},
\begin{equation}\label{3Varvak}
\sum_{B \subset [m-\ell ] \times [\ell ]} r_k^{(3)}(B)= \binom{m}{m-\ell+k} a(m-\ell+k,m-\ell).
\end{equation}
Using \eqref{BesselExp}, one finds the explicit expression
\begin{equation}\label{ExpJordCoeff}
\left\{ \left\{ {m \atop \ell } \right\} \right\}_k = \frac{m! (m-\ell+2k-1)!}{2^k (\ell-k)!k!(m-\ell+k)!(m-\ell-1)!}.
\end{equation}

Now, we consider $s=3$, $p=1$ and $q\neq 1$, that is, the $q$-analog to the above. Thus, we consider the commutation relation 
\begin{equation}\label{s3qanalog}
XY-qYX=hY^3.
\end{equation} 
By defining in analogy to above
\begin{equation}\label{3qBinFile}
\left\{ \left\{ {m \atop \ell } \right\} \right\}_{k|q} := \sum\limits_{\lambda \in \mathcal{I}_{m-\ell,\ell}}R_{3,1;q}[B_{\lambda}, k],
\end{equation}
then $\left\{ \left\{ {m \atop \ell } \right\} \right\}_{k|q} ={\mathcal J}_{3,1,q}(m,\ell,k)$ by Corollary~\ref{qBinomialRook} for $s=3$. Thus, we have the following result.
\begin{corollary}\label{3qBinCor} Let $s=3, p=1,q\neq 1$. Then $Z_1=I$ and \eqref{pqdefgenWeyl} reduce to the commutation relation $XY=qYX+hY^3$ and one has 
\begin{equation}\label{3qBinCorFor}
(X+Y)^m=\sum_{\ell=0}^m \sum_{k=0}^{\ell} h^{k} \left\{ \left\{ {m \atop \ell } \right\} \right\}_{k|q} \,Y^{m-\ell+2k} X^{\ell-k}.
\end{equation}
\end{corollary}
After switching the order of summation and relabelling the indices, we rewrite \eqref{3qBinCorFor} equivalently in a form analogous to \eqref{3Expansion},
\begin{equation}\label{3qBinCorFor2}
(X+Y)^m=\sum_{k=0}^m \sum_{\ell=0}^{m-k} h^{m-k-\ell} \left\{ \left\{ {m \atop m-\ell } \right\} \right\}_{m-\ell-k|q} \,Y^{2m-2k-\ell} X^{k}.
\end{equation}
We can write this equivalently as 
\begin{equation}
(X+Y)^m=\sum_{k=0}^m \binom{m}{k}_q Y^{m-k} \left( \sum_{\ell=0}^{m-k} \binom{m}{k}_q^{-1} \left\{ \left\{ {m \atop k+\ell } \right\} \right\}_{\ell|q} \,(hY)^{\ell} \right) X^{k}.
\end{equation}
Introducing
\begin{equation}\label{QBessel}
\vartheta_{m,k}(x;q):=\sum_{\ell=0}^{m-k} \binom{m}{k}_q^{-1} \left\{ \left\{ {m \atop k+\ell } \right\} \right\}_{\ell|q} \,x^{\ell},
\end{equation}
we can write \eqref{3qBinCorFor} for $h=-1$ in the following equivalent form similar to \eqref{bf2823b}, 
\begin{equation}\label{qbf2823b}
(X+Y)^m=\sum_{k=0}^m \binom{m}{k}_q Y^{m-k}\vartheta_{m,k}(-Y;q)  X^{k}.
\end{equation}
\begin{remark}\label{RemarkQBessel}
A $q$-analog to the Bessel polynomials was introduced by Ismail \cite{MEHI1989}. Abdi \cite{WHA1965} considered another version of the $q$-deformed Bessel polynomials, whereas, Dulucq \cite{SD1992} discussed the $q$-analog to the Bessel polynomials in a combinatorial way, and he gave the following $q$-analog to \eqref{BessPol},
\begin{equation}\label{qBessPol}
y_{n}(x;q)=\sum\limits_{k=0}^{n}\binom{n+k}{n-k}_{q}[2k-1]_{q}!!q^{\binom{n-k}{2}}x^k.
\end{equation}
Thus, comparing \eqref{QBessel} and \eqref{qBessPol}, it would be interesting to find a connection between $\vartheta_{m,k}(x;q)$ and the $q$-deformed Bessel polynomials $y_{\ell}(x;q)$.
\end{remark}
On the other hand, according to \cite[Proposition 5.2]{TMMS2011} (with our notation), we have for variables satisfying \eqref{s3qanalog} the binomial formula,
\begin{equation}\label{d3qBinoFor}
(X+Y)^m=\sum\limits_{k=0}^{m}\bigg(\sum_{i=0}^{m-k-1}h^id_{m}^{(3)}(k,i;q)Y^{m-k+i}\bigg)X^k,
\end{equation}
where, for all $0\leq k\leq i+k\leq n+1$, the numbers $d_m^{(3)}(k,i;q)$ satisfy the following recurrence relation
\begin{equation}\label{dRecurr}
d_{m+1}^{(3)}(k,i;q)=d_{m}^{(3)}(k,i;q)+[m-k+i-1]_{q}d_{m}^{(3)}(k,i-1;q)+q^{m+i-k+1}d_{m}^{(3)}(k-1, i;q),
\end{equation}
with the boundary condition 
$$d_{m+1}^{(3)}(m+1,0;q)=1 \text{ and } d_{m+1}^{(3)}(k,m+1-k;q)=0 \text{ for all } k=0,1,\ldots,m.$$
Hence, by adopting the indices for better comparison with \eqref{3qBinCorFor2}, we can rewrite \eqref{d3qBinoFor} equivalently as follows
\begin{equation}\label{d3qBinoFor2}
(X+Y)^m=\sum\limits_{k=0}^{m}\sum_{\ell=1}^{m-k}h^{m-k-\ell}d_{m}^{(3)}(k,m-k-\ell;q)Y^{2m-2k-\ell}X^k.
\end{equation}
Thus, comparing \eqref{d3qBinoFor2} with \eqref{3qBinCorFor2}, one has
\begin{equation}\label{dqBinoCoeff}
d_{m}^{(3)}(k,m-k-\ell;q)=\left\{ \left\{ {m \atop m-\ell } \right\} \right\}_{m-\ell-k|q}.
\end{equation}
Some index manipulations yield the following proposition.
\begin{proposition}\label{q3BinExP} The coefficients defined in \eqref{3qBinFile} are given as follows
\begin{equation}\label{q3BinEx}
\left\{ \left\{ {m \atop \ell } \right\} \right\}_{k|q} = d_{m}^{(3)}(\ell-k,k;q), 
\end{equation}
where the numbers $d_{m}^{(3)}(k,i;q)$ are given in \eqref{dRecurr}.
\end{proposition}
Combining \eqref{3qBinFile} and \eqref{q3BinEx}, we obtain
\begin{equation}
\sum\limits_{\lambda \in \mathcal{I}_{m-\ell,\ell}}R_{3,1;q}[B_{\lambda}, k]=d_{m}^{(3)}(\ell-k,k;q).
\end{equation}

Finally, for the case $p\neq 1$, we let
\begin{equation}
\left\{ \left\{ {m \atop \ell } \right\} \right\}_{k|p;q} := \sum\limits_{\lambda \in \mathcal{I}_{m-\ell,\ell}} \mathscr{R}_{3, 1;p, q}[B_{\lambda}, k].
\end{equation}
By the same arguments as above, we have the following result.
\begin{corollary}
Let $s=3$ and $ p \neq 1 \neq q$. Then \eqref{pqdefgenWeyl} reduce to the commutation relations $XY=qYX+hY^3, Z_pY=pYZ_p, XZ_p=pZ_pX$, and one has 
\begin{equation}\label{Bin23pq}
(X+Y)^m=\sum_{\ell=0}^m \sum_{k=0}^{\ell} h^{k} \left\{ \left\{ {m \atop \ell } \right\} \right\}_{k|p;q} \,Y^{m-\ell+2k} Z_p^{k}X^{\ell-k}.
\end{equation}
\end{corollary}
However, it seems not easy to find an explicit expression for the coefficients $\left\{ \left\{ {m \atop \ell } \right\} \right\}_{k|p;q}$. Exactly as in the case $p=1$ considered above, we can write \eqref{Bin23pq} equivalently as
\begin{equation}
(X+Y)^m=\sum_{k=0}^m \binom{m}{k}_{p,q} Y^{m-k} \left( \sum_{\ell=0}^{m-k} \binom{m}{k}_{p,q}^{-1} \left\{ \left\{ {m \atop k+\ell } \right\} \right\}_{\ell|p;q} \, h^{\ell}Y^{\ell}Z_p^{\ell} \right) X^{k}.
\end{equation}
Noting that $Y^{\ell}Z_p^{\ell}=p^{- \binom{\ell}{2}}(YZ_p)^{\ell}$, we introduce the following $(p,q)$-analog to \eqref{QBessel} by
\begin{equation}\label{PQBessel}
\vartheta_{m,k}(x;p,q):=\sum_{\ell=0}^{m-k} p^{- \binom{\ell}{2}}\binom{m}{k}_{p,q}^{-1} \left\{ \left\{ {m \atop k+\ell } \right\} \right\}_{\ell|p;q} \,x^{\ell}.
\end{equation}
Thus, we can write \eqref{Bin23pq} for $h=-1$ in the following form,
\begin{equation}\label{pqbf2823b}
(X+Y)^m=\sum_{k=0}^m \binom{m}{k}_{p,q} Y^{m-k}\vartheta_{m,k}(-YZ_p;p,q)  X^{k}.
\end{equation}
Similar to the case $p=1$ (Remark~\ref{RemarkQBessel}) it would be interesting to find the relation between $\vartheta_{m,k}(x;p,q)$ defined above and the $(p,q)$-deformed Bessel polynomials introduced in literature (e.g., in \cite{GYAM2021}).

\section{Conclusion}\label{concl}
Following our introduction of normal ordering in the $(p, q)$-deformed generalized Weyl Algebra in \cite{TLM1,TLM2} -- focusing first on its algebraic framework and subsequently on its interpretation via rook placements -- the present work establishes further explicit normal ordering expansions within this setting. Specifically, the binomial formula $(X+Y)^n$ was derived within this deformed algebra. Upon expanding the power one has to be careful due to the noncommutativity of the letters $X,Y $ and $Z_p$. To fix the order of the letters on the right-hand side, we chose to use the normal ordered form (as is usual). Thus, to derive the binomial formula, we had to perform two steps: 1) Write $(X+Y)^n$ as a sum over all possible words of length $2n$ containing $n$ letters of $X$ and $n$ letters of $Y$. This is standard and quite generic. 2) Normal order each of these words in the sum. This depends on the commutation relations assumed and involved in our case the $(p,q)$-deformed $s$-rook numbers introduced before. In the special case of the conventional Weyl algebra (i.e., $s=0$ and $p=1=q$), Varvak \cite{AV2005} gave a combinatorial interpretation of the normal ordering coefficients of $(X+Y)^n$ in terms of a sum of rook numbers over all Ferrers boards fitting into particular rectangles. We derived for small values of $s$ ($s=0,1,2,3$) and for special values of the parameters $p,q,h$ analogous results, thereby recovering known results in the literature as well as deriving new ones. However, for some cases (in particular when $p\neq 1 \neq q$, but sometimes also when $p=1, q\neq 1$) we were not able to give explicit expressions -- and these remain open for future study.

\end{document}